\documentclass[reqno]{amsart}
\usepackage{amssymb,graphicx}
\usepackage{tikz}
\usetikzlibrary{shapes,arrows}
\usetikzlibrary{shapes.geometric,arrows,chains}
\usepackage{mathrsfs}
\usepackage{ifpdf}
\usepackage[hidelinks]{hyperref}
\hypersetup{colorlinks=true,linktocpage=true,
	linkcolor=blue,
	filecolor=blue,
	citecolor = blue,
	urlcolor=red}
\usepackage{setspace}
\usepackage[fontsize=11.5pt]{fontsize}
\usepackage{bm}

\usepackage{geometry}
\allowdisplaybreaks[4]
\theoremstyle{plain}
\newtheorem*{theorem}{Theorem}
\newtheorem{thm}{Theorem}[section]
\newtheorem{lem}[thm]{Lemma}
\newtheorem{cor}{Corollary}[section]
\newtheorem{definition}[thm]{Definition}

\newtheorem{rmk}[thm]{Remark}

\numberwithin{equation}{section}

\def\a{\alpha}
\def\b{\beta}
\def\d{\Delta}
\def\e{\epsilon}
\def\g{\gamma}
 
\def\CD{${\rm CD}(K,N)$}
\def\CDS{${\rm CD}^{\ast}(K,N)$}
\def\G{Geo(X)}
\def\l{\lambda}
\def\M{\mathbf{M^n}}
\def\O{\Omega}
\def\P{\bm{\pi}}
\def\p2{\mathscr{P}_2(X)}
\def\PC{\mathscr{P}(C([0,1],X))}
\def\L{\rm{LIP(X)}}
\def\LC{\rm LIP_{0}(X)}

\def\R{\mathbb{R}}
\def\RCD{${\rm RCD}^{\ast}(K,N)$}
\def\RCDn{${\rm RCD}^{\ast}(-K,N)$}
\def\RCDI{${\rm RCD}(K,\infty)$}
\def\RCDO{${\rm RCD}^{\ast}(0,N)$}
\def\W{W^{1,2}}
\def\X{$(X,d,\mu)$}

\makeatother
\makeatletter
\@namedef{subjclassname@2020}{\textup{2020} Mathematics Subject Classification}
\makeatother

\begin{document}

\title[Semilinear elliptic equations on metric measure spaces]
{gradient estimate and Universal bounds for semilinear elliptic equations on RCD$^*$(K,N)\\ metric measure spaces}

\author[Zhihao Lu]{Zhihao Lu}
\address[Zhihao Lu]{School of Mathematics and Statistics, Jiangsu Normal University, Xuzhou 221116, P. R. China}
\email{zhihaolu@jsnu.edu.cn}

\begin{abstract}
We derive logarithmic gradient estimate and universal boundedness estimate for semilinear elliptic equations on \RCD\, metric measure spaces, which contains the class of Riemannian manifolds with Ricci curvature bounded below. These estimates are applicable for equations satisfying subcritical index condition,which recover many classical results even on Euclidean spaces. In certain case, these estimates are optimal even on \RCD\,\,spaces with $K<0$. Two direct corollaries of these estimates are Harnack inequality and Liouville theorem. In addition to these estimates, we also establish fundamental relations among the universal boundedness estimate, the logarithmic gradient estimate, and Harnack inequality. Under certain and wild assumptions for the nonlinear term, we prove that these estimates are $\kappa$-equivalent on \RCDO\,spaces for any $\kappa>1$.
\end{abstract}


 
\thanks{School of Mathematics and Statistics, Jiangsu Normal University, Xuzhou 221116, P. R. China}


\maketitle
{\tableofcontents}
\section{Introduction}

\subsection{Motivation}
The following Lane-Emden equation ($\a\in\mathbb{R}$)
\begin{equation}\label{LE}
	\Delta u+u^{\alpha}=0
\end{equation}
is an important model of semilinear elliptic equations. In 1981, Gidas-Spruck \cite{GS} got the following breakthrough for this equation under subcritical index condition. 
\begin{theorem}\cite[Theorem 3.1,Theorem 6.1]{GS}
	Let $(\M,g)$ {\rm ($n\ge3$)} be an n-dimensional Riemannian manifold and $u$ be a non-negative $C^2$ solution of \eqref{LE} in $\O\subset\M$ with $\a\in(1,\frac{n+2}{n-2})$, then\\
	{\rm(I) (Liouville theorem)}  if $Ricci(\M)\ge 0$ and $\O=\M$, we have $u\equiv 0$.\\
	{\rm(II) (Singular decay estimate)}  if $\M=\mathbb{R}^n$ and $\O=B(0,R)\setminus\{0\}$, we have
	\begin{equation}
		u(x)\le c|x|^{-\frac{2}{\a-1}},\qquad\text{as}\quad x\to 0.
	\end{equation}\\
{\rm(III) (Harnack inequality)}  if $\M=\mathbb{R}^n$ and $\O=B(0,R)\setminus\{0\}$, then there exists $\theta_0$ such that for any $0<\theta\le\theta_0$ and $0<\e\le\frac{1}{2}R$,
\begin{equation}
	\sup\limits_{\e\le|x|\le (1+\theta)\e}u(x)\le c\inf\limits_{\e\le|x|\le (1+\theta)\e}u(x),
\end{equation}
where $c$ is independent of $u$, $\e$, $\theta$.
\end{theorem}

To our knowledge, other than their original proof based on integral estimates (also called integral Bernstein method in \cite{BGV}), there is essentially no novel proof available for their Liouville theorem (I). If one assumes the domain manifold as $\mathbb{R}^n$, then other methods can be found to reprove it, such as the moving plane or moving sphere methods (see \cite{CL,D1,LZ} and references therein). For $\alpha\in(0,1]$, Mitidieri–Pohozaev \cite{MP98} obtained the Liouville theorem even for quasilinear elliptic inequalities.
  In 1991, using integral estimates as well, Bidaut-Véron–Véron \cite{BV} obtained the classification result for semilinear equations on compact manifolds. Except for the Lane–Emden equation, Gidas and Spruck \cite{GS} also generalized the above results to the following semilinear equation
\begin{equation}\label{GE}
	\Delta u+f(u)=0
\end{equation}
with some additional conditions for the nonlinearity
$f$.
Until 1998, Dancer \cite{D2} found that the estimate in part {\rm (II)} of above theorem is universal\footnote{The word "universal" came from Serrin-Zou \cite{SZ}, which means that the estimate is not only independent of any given solution but also do not require any boundry conditions whatsoever. } . Concretely, he got
\begin{theorem}\cite[Lemma 1]{D2}
	Let $u$ be a  non-negative $C^2$ solution of \eqref{LE} in a domain $\O\subsetneqq\mathbb{R}^n$, with $\a\in(1,\frac{n+2}{n-2})$ and $n\ge 3$. Then 
	\begin{equation}
		u(x)\le c(n,\a)[dist(x,\partial\O)]^{-\frac{2}{\a-1}}.
	\end{equation}
\end{theorem}

Later, by establishing an important generalized Gidas–Spruck's identity,
  Serrin-Zou \cite{SZ} extended Gidas-Spruck's Liouville theorem (in Euclidean space), Harnack inequality, and Dancer's universal boundedness estimate to quasilinear elliptic equations. In 2007, using the celebrated doubling lemma and blow-up method, Poláčik--Quittner--Souplet \cite{PQS} obtained a universal boundedness estimate for solutions of general elliptic equations and systems under the assumptions of the Liouville theorems. Therefore, they put forth the perspective that the Liouville theorem is equivalent to the universal boundedness estimate, see also \cite{GS0,QS}.  For more quasilinear equations, even those with a gradient structure, one can also derive a priori estimates for solutions or their gradients via the direct Bernstein method; see Lions \cite{Lions}, Bidaut-Véron--García-Huidobro--Véron \cite{BGV,BGV20}, Filippucci--Sun--Zheng \cite{FSZ24}, and the references therein.

Now, let us briefly review their methods of obtaining universal bounds for \eqref{LE} (or \eqref{GE}). According to Serrin-Zou \cite{SZ}, the crucial Harnack inequality for equation \eqref{GE} can be derived using the well-known De Giorgi-Nash-Moser iteration method. This leads to obtaining universal bounds for positive solutions through Pohozaev's method. However, when attempting to generalize their results to equation \eqref{GE} on manifolds or metric measure spaces, the Harnack constant remains unclear. This lack of clarity significantly impacts the application of Pohozaev's method and fails to demonstrate how this bound depends on the spaces. Similarly, Poláčik-Quittner-Souplet's blow-up technique  also falls short in providing a clear bound on the aforementioned setting. Therefore, in order to obtain the universal boundedness estimate and Harnack inequality for equation \eqref{GE} on metric measure spaces, we need to explore new approaches.

Due to the aforementioned reasons, we outline our objectives as follows. In this paper, we establish a logarithmic gradient estimate and a universal boundedness estimate for equation \eqref{GE} on a metric measure space, which also implies the Harnack inequality. To achieve a clear understanding of the dependence, we require that the metric measure spaces are what is known as \RCD\, metric measure spaces (\RCD\,spaces for short). The theory of \RCD\,spaces has been developed by numerous experts over the past two decades. Before presenting our results, we would like to emphasize that Bochner's inequality is crucial for all proofs, and thus we select \RCD\,spaces as our domain spaces where Bochner's inequality holds. In the following paragraphs, we provide a brief overview of the development of the theory (related definitions can be found in Section 2).

In general, the lower bound of Ricci curvature plays a crucial role in determining various analytic and geometric properties of Riemannian manifolds. However, Riemannian manifolds with Ricci curvature bounded below are not stable under Gromov-Hausdorff convergence. To address this issue, Strum \cite{St1,St2} and Lott-Villani \cite{LV2} introduced the \emph{curvature dimensional condition} for metric measure spaces, which are Polish spaces with a positive Borel measure. For a given $K\in\R$ and $N\in[1,\infty]$, the \CD\,space represents the class of metric measure spaces that have a ``general Ricci curvature" bounded below by $K\in\R$ and a ``dimension" above by $N$. This notation aligns with the classical notation used in the smooth setting. Specially, a Riemannian manifold has dimension less or equal to $N$ and Ricci curvature greater or equal to $K$ if and only if it is a \CD\,space. The \CD\,space class is stable under measured Gromov-Hausdorff convergence and possesses several fundamental properties, such as Bishop-Gromov volume growth, Bonnet-Meyers diameter bound, the Lichnerowicz spectral gap, Brunn-Minkowski inequality, ect. On the other hand, to establish the validity of essential properties such as the local-to-global and tensorization, Bacher-Sturm \cite{BS} introduced the \emph{reduced curvature dimension condition}, denoted as \CDS, which shares the same favorable geometric features of \CD.  

However, both \CD\,and\,\,\CDS\, conditions include Finsler geometry, where the validity of the important Bochner's inequality  remains unknown. In order to exclude Finsler geometry, Ambrosio-Gigli-Savar\'{e} \cite{AGS1,AGS2,AGS3} made significant contributions by developing the \emph{Riemannian curvature dimension condition} \RCDI, which requires the space to be infinitesimally Hilbertian. The finite dimensional Riemannian curvature dimension condition, \RCD, was later introduced by Gigli \cite{G} and Erbar--Kuwada--Sturm \cite{EKS}. 
Relying in part on \cite{AGS1}, which treats the case $N = \infty$, the authors demonstrated that the \RCD\,condition implies Bochner's inequality $BE(K,N)$ and Bakry--Ledoux pointwise gradient estimate $BL(K,N)$, and these conditions are equivalent under certain mild regularity assumptions. Similar charaterizations were also studied in Ambrosio--Mondino--Savar\'{e} \cite{AMS} by a different approach. As a corollary of their results, the class of \RCD\,spaces contains (weighted) Riemannian manifolds with (Bakry-\'{E}mery) Ricci curvature $\ge K$ and dimension $\le N$ as well as $N$ dimensional Alexandrov spaces with generalized sectional curvature bounded below (see Zhang--Zhu \cite{ZZ1}). 


\subsection{Main results}
For any real number $N\in[1,\infty)$, we define Sobolev constant as follows
\begin{eqnarray}
	p_S(N)=\begin{cases}
	\infty    \quad\qquad\qquad\,\text{if}\quad N\in[1,2]\nonumber\\
	\frac{N+2}{N-2}\qquad\qquad\,\,\text{if}\quad N\in(2,\infty).\nonumber
	\end{cases}
\end{eqnarray}
For deriving logarithmic gradient estimate for more general semilinear elliptic equations such as \eqref{GE}, we also need to consider the following index:
\begin{eqnarray}
p(N)=\begin{cases}
\infty    \quad\qquad\qquad\,\text{if}\quad N=1\nonumber\\
\frac{N+3}{N-1}\qquad\qquad\,\,\text{if}\quad N\in(1,\infty).\nonumber
\end{cases}
\end{eqnarray}

Now, let us establish some conventions. We denote a metric measure space as $(X,d,\mu)$, where $(X,d)$ is a Polish space and $\mu$ is a non-negative Borel measure. Throughout this paper, we assume that the support of $\mu$ is equal to $X$, that $\mu$ is finite on bounded sets, and that $(X,d)$ is a proper metric space. We use $\L$ to represent  the space of Lipschitz functions and  $\LC$ to represent  the space of Lipschitz functions with compact support. Additionally, $\mathscr{P}(X)$ denotes the space of Borel probability measures on $(X,d)$. The notation $B(x,R)$ or $B_{R}(x)$ signfies an open ball with center $x$ and radius $R$. The notation $\chi_E$ refers to the characteristic function of the set $E$. The symbol $\sup_{E}g$ ($\inf_{E}g$) represents the essential supremum (infimum) of a measurable function $g$ over a measurable set $E$. The nonlinear term $f$ in equation \eqref{GE} is always assumed to be $C^2$ on $(0,\infty)$. Furthermore, all solutions discussed in the paper are weak solutions (as defined in Definition \ref{ws}). For a positive  function $f\in C^2(0,\infty)$, we define
\begin{equation}
	\l=\inf\limits_{t\in(0,\infty)}\frac{tf'(t)}{f},
\end{equation}
\begin{equation}
	\Lambda=\sup\limits_{t\in(0,\infty)}\frac{tf'(t)}{f},
\end{equation}
\begin{equation}
	\Pi=\inf\limits_{t\in(0,\infty)}\frac{t^2f''(t)}{f}.
\end{equation}

Now, we state our main results. First, for simplest Lane-Emden equation \eqref{LE}, we have the following optimal logarithmic gradient estimate and universal bounds, which recovers (I), (II) and (III) in Gidas and Spruck's theorems simultaneously.
\begin{thm}\label{sthm}
	\rm	Let $(X,d,\mu)$ be a \RCDn\, space with $N\in[1,\infty)$ and $K\ge 0$. If $u$ is a positive solution of \eqref{LE} on $B(x_0,2R)$ with $\a\in(-\infty,p_S(N))$, then 
	\begin{equation}\label{strong0}
		\sup\limits_{B(x_0,R)}\left(\frac{|\nabla u|^2}{u^2}+u^{\alpha-1}\right)\le C\left(K+\frac{1}{R^2}\right),
	\end{equation}	
	where $C=C(N,\a)$.
	
\end{thm}

In fact, Theorem \ref{sthm} is a corollary of Theorem \ref{main} and Theorem \ref{S}. It is worth noting that Theorem \ref{sthm} is not only optimal when the curvature is non-negative, but also when the curvature is negative. This point will be illustrated by the examples in the section \ref{S9}. Our frame of proof of Theorem \ref{sthm} is a measure-theoretic version of the Bernstein method; see \cite{ZZ2} for harmonic function and heat equation.
For the partially subcritical range, many previous works derived the Cheng–Yau inequalities (from Yau \cite{Y} and Cheng–Yau \cite{CY}) for \eqref{LE} on manifolds, see \cite{HWW, Li91, LU1, PWW, WW, Yang} and the references therein.

For more general positive nonlinear term $f$, we can obtain the logarithmic gradient estimate and universal boundedness estimate without restrictions on the lower index $\lambda$. 

\begin{thm}\label{main}
	Let $(X,d,\mu)$ be a \RCDn\, space with $N\in[1,\infty)$ and $K\ge 0$. If $u$ is a positive solution of \eqref{GE} on $B(x_0,2R)\subset X$ and the term $f$ satisfies\\
	{\rm (1)} $f(t)>0$ on $(0,\infty)$,\\
	{\rm (2)} $\Lambda<p(N)$,\\
	{\rm (3)} $\Pi>-\infty$,\\
	 then we have
	\begin{equation}\label{strong}
	\sup\limits_{B(x_0,R)}\left(\frac{|\nabla u|^2}{u^2}+\frac{f(u)}{u}\right)\le C\left(K+\frac{1}{R^2}\right),
	\end{equation}	
	where $C=C(N,\Lambda,\Pi)$.
	
\end{thm}

From Theorem \ref{main}, we immediately have the following corollary.
\begin{cor}
	Let $(X,d,\mu)$ be a \RCDn\, space with $N\in[1,\infty)$ and $K\ge 0$. If $f(t)$ satisfies \\
	{\rm (1)} $f(t)>0$ on $(0,\infty)$,\\
	{\rm (2)} $\Lambda<p(N)$,\\
	{\rm (3)} $\Pi>-\infty$,\\
	then we  have \\
	{\rm (i)} if $K=0$, equation \eqref{GE} does not have positive solutions on $X$;\\
	{\rm (ii)} if $M_0=\inf\limits_{(0,\infty)}\frac{f(t)}{t}>0$, then there exist positive numbers $K_0=K_0(N,\Lambda,\Pi,M_0)$ and $R_0=R_0(N,\Lambda,\Pi,M_0)$ such that if $K=K_0$, then equation \eqref{GE} does not have positive solutions on $B(x_0,R)$ for any $x_0\in X$ and $R\ge R_0$.
\end{cor}
A common nonlinear term that satisfies conditions $(1)$--$(3)$ and (ii) is $ku^{\a}+lu^{\b}$, where $\b<1<\a<p(N)$ and $k,l>0$. In this case, the aforementioned corollary implies that equation \eqref{GE} does not even have positive solutions in balls with finite radii when $K$ is small.

If we remove the positivity assumption of $f$  in Theorem \ref{main}, we can still obtain the logarithmic gradient estimate (and consequently the Harnack inequality) for equation \eqref{GE}. However, in this case, we cannot obtain information regarding the boundedness of solutions. Concretely, we have
\begin{thm}\label{main1}
	Let $(X,d,\mu)$ be a \RCDn\, space with $N\in[1,\infty)$ and $K\ge 0$. We assume that the nonlinear term $f$ satisfies $t^{-\a}f(t)$ is non-increasing on $(0,\infty)$ for some $\a\in(1,p(N))$.
	If $u$ is a positive solution of \eqref{GE} on $B(x_0,2R)\subset X$, then we have
	\begin{equation}\label{weak}
	\sup\limits_{B(x_0,R)}\frac{|\nabla u|^2}{u^2}\le C\left(K+\frac{1}{R^2}\right),
	\end{equation}	
	where $C=C(N,\a)$.	
\end{thm}

A corollary of Theorem \ref{main1} is the following Liouville property. 
\begin{cor}\label{cc}
	Let $(X,d,\mu)$ be a \RCDO\,space with $N\in[1,\infty)$.
	If for some $\a\in(1,p(n))$, $t^{-\a}f(t)$ is non-increasing on $(0,\infty)$. Then any positive solution of \eqref{GE} on $X$ must be constant.
\end{cor}
A common nonlinear term that satisfies the condition in Corollary \ref{cc} is $ku-lu^d$, where $k,l\ge0$ and $d>1$ without any dimension restrictions. In Section 8, we will further relax the curvature condition for such nonlinear equations.

If $\Lambda\ge p(N)$, we need additional restrictions on the nonlinear term compared to those in Theorem \ref{main}. We first present the following weak form of the logarithmic gradient estimate and universal boundedness estimate,  even with slightly strengthened  conditions.

\begin{thm}\label{main2}
	Let $(X,d,\mu)$ be a \RCDn\,space with $N\in[1,\infty)$ and $K\ge 0$. If $u$ is a positive  solution of \eqref{GE} on $B(x_0,2R)\subset X$  and the nonlinear term $f$ satisfies\\
	{\rm (1)} $f(t)>0$ on $(0,\infty)$,\\
	{\rm (2)} $\Lambda\in[p(N),p_S(N))$,\\
	{\rm (3)} $\frac{tf'(t)}{f}$ is non-decreasing on $(0,\infty)$,\\
	{\rm (4)} if $N\in [4,\infty)$,  $\lambda\ge 1$; if $N\in(1,4)$,  $\lambda>2$,\\
	then we have\\
	{\rm (I)} if $N\in[4,\infty)$, there exist $\b=\b(N,\Lambda)>0$ and $L=L(N,\Lambda)$ such that for any $\e>0$,
	\begin{equation}\label{ss1}
	\sup\limits_{B(x_0,R)}(u+\e)^{-\b}\left(\frac{|\nabla u|^2}{u^2}+\frac{f(u)}{u}\right)\le C(N,\Lambda)\e^{-\b}\left(K+\frac{1}{R^2}+\frac{f(L\e)}{\e}\right),
	\end{equation}
	{\rm (II)} if $N\in[1,4)$, there exist $\b=\b(N,\Lambda)>0$ and $L=L(N,\lambda,\Lambda)$ such that for any $\e>0$,
	\begin{equation}\label{ss2}
	\sup\limits_{B(x_0,R)}(u+\e)^{-\b}\left(\frac{|\nabla u|^2}{(u+\e)^2}+\frac{f(u)}{u}\right)\le C(N,\lambda,\Lambda)\e^{-\b}\left(K+\frac{1}{R^2}+\frac{f(L\e)}{\e}\right).
	\end{equation}
\end{thm}

In order to remove $\e$ in Theorem \ref{main2} and get a clearer estimate same as that in  Theorem \ref{main}, we require the following definition. 
\begin{definition}\label{d0}
	An increasing function $f$ on $(0,\infty)$ is called $h$-inverse bounded for some positive function $h$ on $(0,\infty)$ if $f(t)\le Cf(\e)$ implies $t\le h(C)\e$ for any $C,\e>0$.
\end{definition}

In addition to Definition \ref{d0}, we also need two more conditions to obtain the strong form of the logarithmic gradient estimate and universal boundedness estimate. These additional conditions can be considered as superlinear conditions in some sense. However, in this case, we require further restrictions on the increment speed of the nonlinear term.

\begin{thm}\label{S}
Let $(X,d,\mu)$ be a \RCDn\,space with $N\in[1,\infty)$ and $K\ge 0$. If $u$ is a positive solution of \eqref{GE} on $B(x_0,2R)\subset X$ and $f$ satisfies {\rm (1)--(4)} in Theorem \ref{main2} and\\ 
	{\rm (V1)} $\lim\limits_{t\to 0^{+}}t^{-1}f(t)=0$,\\
	{\rm (V2)} there exists $\b\in(\rho(N,\Lambda),\infty)$ such that $t^{-1-\b}f(\cdot)$ is $h$-inverse bounded for some positive function $h$, where 
	\begin{eqnarray}
		\rho(N,\Lambda)=\begin{cases}
			\frac{2N}{N+4}(\Lambda-1-\frac{2}{N})\qquad\qquad\,\,\text{if}\quad N\in(1,2]\,\,\text{or}\,\, N\in(2,3)\, \text{and}\,\, \Lambda<\frac{N+1}{N-2}	\nonumber\\
			\frac{2(N-1)}{N+2}\Lambda-2\qquad\qquad\qquad\,\text{if} \quad N\in(2,3)\, \text{and}\,\, \Lambda\ge\frac{N+1}{N-2}\,\,\text{or}\,\, N\in[3,\infty),
		\end{cases}
	\end{eqnarray}
	then we have
	\begin{equation}\label{sss}
		\sup\limits_{B(x_0,R)}\left(\frac{|\nabla u|^2}{u^2}+\frac{f(u)}{u}\right)\le C\left(K+\frac{1}{R^2}\right),
	\end{equation}		
	where $C=C(N,\lambda,\Lambda,\b)$.
\end{thm}
Notice that $1+\rho(N,\Lambda)<\Lambda$ when $\Lambda\in[p(N),p_S(N))$. Therefore, for equation \eqref{LE}, when $\a\in [p(N),p_S(N))$, the assumptions of Theorem \ref{S} are valid and so estimate \eqref{sss} holds.

\begin{rmk}
	\rm 
    We suspect that $f(u)=u^{\a}+u^{\delta}$  with $-\infty<\a,\delta<p_{S}(N)$ (and general $\sum_{i=1}^{m}k_iu^{\a_i}$ with $k_i>0$ and $\max(\a_i)<p_S(N)$) also satisfies estimate \eqref{sss}. To our knowledge, it is unknown for such concise nonlinear equation even on Euclidean spaces.
	
\end{rmk}

\textbf{Organization of the paper.} In Section 2, we introduce the necessary knowledge of \RCD\,spaces and the basic setting for our problem. Section 3 provides two types of auxiliary functions  and their elliptic inequalities, derived using Bochner's inequality. These results serve as the foundation for the subsequent arguments. Additionally, we also discuss the use of cut-off function and the weak maximal principle within current framework.  In Section 4 and Section 5, we utilize the two special auxiliary functions constructed in Section 3 to prove Theorem \ref{main} and Theorem \ref{main1}, respectively. In Section 6, by leveraging the results from Section 2 and 3, we establish the relations among three fundamental estimates: universal boundedness estimate, logarithmic gradient estimate, and Harnack inequality. Notably, under certain assumptions, these estmates are $\kappa$-equivalent on \RCDO\,spaces for any $\kappa>1$. Partial results from Section 6 are then utilized in the proofs of Theorem \ref{main2} and \ref{S} in Section 7, specially when  considering equations in spaces with "dimension" less than four.  
In Section 8, we  improve upon Theorem \ref{main1} for Lichnerowicz type equation and obtain the logarithmic gradient estimate and its corresponding Liouville theorem on \RCDn\,spaces for certain $K>0$. Lastly, in the appendix, we present an example that serves to illustrate the optimality of Theorem \ref{sthm} in its global form for the $N>3$ case.

\section{Preliminary on \RCD\, space}

\subsection{Sobolev space on metric measure spaces}
Given an open interval $I\subset \R$, an exponent $p\in[1,\infty]$, and $\gamma:I\to\R$, we say $\gamma$ belongs to $AC^p(I;X)$ if 
\begin{equation}
d(\gamma_s,\gamma_t)\le \int_{s}^{t}g(r)dr \quad\forall s,t\in I, s<t\nonumber
\end{equation}
for some $g\in L^p(I)$. The case $p=1$ corresponds to absolutely continuous curves. A continuous curve $\gamma: [0,1]\to X$ is said to be a geodesic provided 
\begin{equation}
d(\gamma_s,\gamma_t)=|s-t|d(\gamma_0,\gamma_1) \quad\forall s,t\in[0,1].\nonumber
\end{equation}
$(X,d)$ is said a geodesic space if for any $x_0,x_1\in X$, there exists a geodesic $\gamma$ joining $x_0$ and $x_1$. We will denote by $C([0,1],X)$ and $\G$ the space of continuous curves on $[0,1]$ and the space of geodesics, which are both endowed with the supremum norm. For every $t\in[0,1]$, the map $e_t:C([0,1],X)\to X$ is the evaluation at time $t$ defined by $e_t(\gamma)=\gamma(t)$.

Then we start to describe the theory of \emph{Sobolev space}. There are several different approaches to the theory of weakly differentiable functions over metric measure space such as \cite{AGS2,C,H,HK,Sh}. Among them, we follow the one based upon the concept of test plan in Ambrosio-Gigli-Savar\'{e} \cite{AGS2}. In \cite{AGS2}, authors had proved that their definition of \emph{Sobolev space} $W^{1,2}(X)$ are equivalent to Shanmugalingam's definition in \cite{Sh} and Cheeger's definition in \cite{C}. 
\begin{definition}
	A probability measure $\P\in\PC$ is said to be a test plan on $X$ provided the following two properties are satisfied:\\
	{\rm(i)} There exists a constant $C>0$ such that $(e_t)_{\sharp}\P\le C\mu$.\\
	{\rm(ii)} $\P$ is concentrated on $AC^2([0,1],X)$ and it holds that 
	\begin{equation}
	\int\int_{0}^{1}|\dot{\gamma}_t|^2dtd\P(\gamma)<\infty.\nonumber
	\end{equation}
\end{definition}
\begin{definition}
	The Sobolev class $S^2(X)$ is defined as the space of all Borel functions $f: X\to\R$ that satisfy the following property: there exists a function $G\in L^2(\mu)$ with $G\ge 0$ such that 
	\begin{equation}
	\int|f(\gamma_1)-f(\gamma_0)|d\P(\gamma)\le\int_{0}^{1}\int G(\gamma_t)|\dot{\gamma}_t|^2d\P(\gamma)dt.\nonumber
	\end{equation}
	Any such $G$ is said to be a weak upper gradient for $f$.
\end{definition}
The set of all weak upper gradient of $f$ is closed and convex in $L^2(\mu)$. Then the unique weak upper gradient of $f$ having minimal $L^2(\mu)$-norm is called \emph{minimal weak upper gradient} of $f$ and is denoted by $|\nabla f|$. We then define the \emph{Sobolev space}
$W^{1,2}(X)$(or $H^1(X)$) as $S^2(X)\cap L^2(X)$ equipped with the norm
\begin{equation}
\parallel f\parallel_{W^{1,2}(X)}=\sqrt{\parallel f\parallel_{L^2(X)}^2+\parallel |\nabla f|\parallel_{L^2(X)}^2}.\nonumber
\end{equation}
For our description, we also need the notion $W^{1,2}(\Omega)$ for some open set $\Omega\subset X$.
\begin{definition}
	Let $\Omega\subset X$ be an open set. A Borel function $f: \Omega\to \R$ belongs to $S^2_{loc}(\Omega)$, provided, for any Lipschitz function $\phi: X\to \R$ with $supp(\phi)\subset \Omega$, it holds $\phi f\in S^2(X)$. Given any $f\in S^2_{loc}(\Omega)$, we define the function $|\nabla f|$ as 
	\begin{equation}
	|\nabla f|=|\nabla(f\phi)|,\quad\mu-a.e.\text{ on}\quad\{\phi=1\},\nonumber
	\end{equation}
	for any $\phi$ as above. The space $S^2(\Omega)$ is the collection of such $f$  with $|\nabla f|\in L^2(\Omega)$.
\end{definition}
The  well-posedness of Definition 3.3 stems from the locality property of minimal weak upper gradients. As before, we can define $W^{1,2}(\Omega)=S^2(\Omega)\cap L^2(\Omega)$. The Sobolev space $W^{1,2}(X)$ is a Banach space, but in general it is not a Hilbert space. If it is a Hilbert space, then the metric measure space $(X,d,\mu)$ is said \emph{infinitesimally Hilbertian}.

Given any two Sobolev functions $f,g\in S^2(X)$, let's define
\begin{equation}
H_{f,g}(\epsilon)=\frac{1}{2}|\nabla(f+\epsilon g)|^2\in L^1(\mu)\quad \text{for every}\,\epsilon\in\R.
\end{equation}
Then the map $H_{f,g}:\R\to L^1(\mu)$ is convex, meaning that
\begin{equation}
H_{f,g}(\lambda \epsilon_0+(1-\lambda)\epsilon_1)\le \lambda H_{f,g}(\epsilon_0)+(1-\lambda)H_{f,g}(\epsilon_1)\quad\mu- a.e.\,\,\text{for all}\,\epsilon_0,\epsilon_1\in\R,\,\lambda\in[0,1].\nonumber
\end{equation}
Therefore the  monotonicity of the difference quotients of $H_{f,g}$ grants that
\begin{equation}
\left\langle \nabla f,\nabla g\right\rangle:= \inf_{\epsilon>0}\frac{H_{f,g}(\epsilon)-H_{f,g}(0)}{\epsilon}=\lim\limits_{\epsilon\downarrow 0}\frac{H_{f,g}(\epsilon)-H_{f,g}(0)}{\epsilon}\in L^1(\mu).
\end{equation}
Similarly, one can define $\left\langle \nabla f,\nabla g\right\rangle$ for $f,g\in S^2_{loc}(\Omega)$. If the metric measure space $(X,d,\mu)$ is \emph{infinitesimally Hilbert}, the map from $W^{1,2}(X)\times W^{1,2}(X)$ to $L^1(\mu)$ given by $ (f,g)\mapsto\left\langle \nabla f,\nabla g\right\rangle$ is bilinear, symmetric and satisfies the Cauchy-Schwarz inequality.

\begin{definition}
	Let $(X,d,\mu)$ be a proper infinitesimally Hilbertian space and $\Omega$ is an open set in $X$. Then a function $f\in W^{1,2}_{loc}(\Omega)$ is in $\mathcal{D}_{loc}(\Delta)$ provided there exists $g\in L^2_{loc}(\Omega)$ such that
	\begin{equation}
	\int gh=-\int \left\langle \nabla f,\nabla h\right\rangle\quad\text{for every}\quad h\in W^{1,2}_0(\Omega),
	\end{equation} 	
	where $W^{1,2}_0(\Omega)$ is the set of $W^{1,2}(\Omega)$ with compact support in $\Omega$.
\end{definition}

\subsection{Riemannian curvature-dimension condition}
In this subsection, we recall some basic definitions and properties of space with lower Ricci curvature bounds that we will need later.

Let $\p2$ be the $L^2$-Wasserstein space over metric measure space $(X,d,\mu)$, i.e. the subset of all Borel probability measure $\nu$ satisfying
\begin{equation}
\int d^2(x,x_0)d\nu<\infty\nonumber
\end{equation}
for some $x_0\in X$. For $\mu_0$,$\mu_1\in\mathscr{P}_2(X)$, the quadratic transportation distance $W_2(\mu_0,\mu_1)$ is defined by
\begin{equation}
W_2^2(\mu_0,\mu_1)=\inf \int_{X\times X}d^2(x,y)dq(x,y)\nonumber
\end{equation}
where the infimum is taken over all $q\in\mathscr{P}(X\times X)$ with $\mu_0$ and $\mu_1$ as the first and second marginal. 
Such a $q$ is also called a $coupling$ of $\mu_0$ and $\mu_1$. Any coupling which realizes the Wasserstein distance is called an \emph{optimal coupling} of $\mu_0$ and $\mu_1$. We denote by $OptGeo(\mu_0,\mu_1)$ the space of all $\P\in\mathscr{P}(Geo(X))$ for which $(e_0,e_1)_{\sharp}\P$ is an optimal coupling. Let $\mathscr{P}_2(X,d,\mu)\subset\p2$ be the subspace of all measures absolutely continuous w.r.t. $\mu$. Denote by $\mathscr{P}_{\infty}(X,d,\mu)$
the set of measures in $\mathscr{P}_{2}(X,d,\mu)$ with bounded support. Then we introduce the \emph{reduced curvature-dimension condition} from \cite[Definition 3.9]{EKS}.

\begin{definition}
	Let $K\in\R$ and $N\in[1,\infty)$. We say that a metric measure space $(X,d,\mu)$ satisfies the reduced curvature-dimension condition ${\rm CD}^*(K,N)$ if and only if for each pair $\mu_0=\rho_0 \mu$, $\mu_1=\rho_1\mu\in\mathscr{P}_{\infty}(X,d,\mu)$ there exist an optimal coupling $q$ of them and a geodesic $(\mu_t)_{t\in[0,1]}$ in $\mathscr{P}_{\infty}(X,d,\mu)$  connecting them such that for all $t\in[0,1]$ and $N'\ge N$
	\begin{equation}
	\int \rho_t^{-\frac{1}{N'}}d\mu_t\ge \int_{X\times X}\Big[\sigma_{K/N'}^{(1-t)}(d(x_0,x_1))\rho_0(x_0)^{-\frac{1}{N'}}+\sigma_{K/N'}^{(t)}(d(x_0,x_1))\rho_1(x_1)^{-\frac{1}{N'}}\Big]dq(x_0,x_1)\nonumber
	\end{equation}
	where the function
	\begin{eqnarray}
	\sigma_k^{(t)}:=
	\begin{cases}
	\frac{\sin(t\theta\sqrt{k})}{\sin(\theta\sqrt{k})}\,\,\,\qquad \text{if}\quad 0<k\theta^2<\pi^2\nonumber\\
	t \quad\qquad\qquad \,\,\,\text{if}\quad k\theta^2=0\nonumber\\
	\frac{\sinh(t\theta\sqrt{k})}{\sinh(\theta\sqrt{k})}\qquad \text{if}\quad 0<k\theta^2<\pi^2\nonumber\\
	\infty\quad\qquad\qquad \text{if}\quad 0<k\theta^2<\pi^2.\nonumber
	\end{cases}
	\end{eqnarray}
	If in addition $(X,d,\mu)$ is infinitesimally Hilbertian, then we say that it is a ${\rm RCD}^*(K,N)$ space.	
\end{definition}
From the work of Erbar-Kuwada-Sturm \cite{EKS}, one has the following Bochner's inequality. 
\begin{thm}\cite[Theorem 4.8]{EKS}\label{B0}
	Let $(X,d,\mu)$ be  a ${\rm RCD}^*(K,N)$ space with $K\in\R$ and $N\in[1,\infty)$. Then for all $f\in\mathcal{D}(\Delta)$ with $\Delta f\in W^{1,2}(X,d,\mu)$ and all $g\in \mathcal{D}(\Delta)\cap L^{\infty}(X,\mu)$ with $g\ge 0$ and $\Delta g\in L^{\infty}(X,\mu)$, we have
	\begin{equation}
	\frac{1}{2}\int\Delta g|\nabla f|^2d\mu-\int g\left\langle \nabla\Delta f,\nabla f\right\rangle\ge K\int g|\nabla f|^2d\mu+\frac{1}{N}\int g(\Delta f)^2d\mu.
	\end{equation}	
\end{thm}
For our application, we need the following measure-theoretic version of Theorem \ref{B0}.
\begin{lem}\cite[Corollary 3.6]{ZZ2}\label{Bochner}
	Let $(X,d,\mu)$ be a ${\rm RCD}^{*}(K,N)$ space with $K\in\mathbb{R}$ and $N\in[1,\infty)$. Assume $f\in \W(B_R(x_0))$ and $\Delta f\in  \W(B_R(x_0))\cap L^{\infty}(B_R(x_0))$. Then we have $|\nabla f|^2\in \W(B_{\frac{3R}{4}}(x_0))\cap L^{\infty}(B_{\frac{3R}{4}}(x_0))$ and that $\Delta |\nabla f|^2$ is a signed Radon measure on $B_{\frac{3R}{4}}(x_0)$. It has the following Radon--Nikodym decomposition 
	\begin{equation}
	\Delta |\nabla f|^2=\Delta^{ac} |\nabla f|^2+\Delta^s |\nabla f|^2,\nonumber
	\end{equation}
	where $\Delta^{ac} |\nabla f|^2$ and $\Delta^s |\nabla f|^2$ are absolutely continuous part and singular part with respect to $\mu$, and
	\begin{eqnarray}
	\Delta^{s} |\nabla f|^2\ge 0,\quad
	\frac{1}{2}\Delta^{ac} |\nabla f|^2\ge\frac{(\Delta f)^2}{N}+\left\langle \nabla \Delta f,\nabla f\right\rangle+K|\nabla f|^2\nonumber
	\end{eqnarray}
	for $\mu$-$a.e.\,\,x\in B_{\frac{3R}{4}}(x_0)$.
\end{lem}

\subsection{Local weak solution of nonlinear equations}
In this subsection, we let $(X,d,\mu)$ be a \RCD\,metric measure space for some $K\in \mathbb{R}$ and $N\in[1,\infty)$. From the viewpoint of weak solution in PDE theory, we introduce the following definition of weak solution of \eqref{GE}.
\begin{definition}\label{ws}
	Let $\Omega\subset X$ be a domain. A function $u(x)$ is called a weak solution of  equation \eqref{GE} on $\Omega$ if $u\in \W_{loc}(\Omega)\cap L^{\infty}_{loc}(\O)$  and
	\begin{equation}
	\int_{\O}-f(u)\cdot \phi+\left\langle \nabla u,\nabla \phi\right\rangle d\mu =0,
	\end{equation}
	for all $\phi\in {\rm LIP}_0(\O)$.\,We say a weak solution is positive if it possesses positive lower bound on any compact subset of $\O$.
\end{definition}
\begin{rmk}
Above definition of positive solution is stronger than the standard sense in PDE theory (positive almost everywhere in the sense of measure theory). This is because when applying the Leibniz rule (Lemma \ref{rule}) to the transformation $w=u^{-\beta}$ with $\beta>0$ (see Section \ref{s3}), we require $w\in L^{\infty}_{\text{loc}}$ (we a priori do not have the crucial Harnack inequality).
By the strong maximum principle \cite[Theorem 2.8]{GR}, if the weak solution is continuous and non-negative with $f\ge 0$, then it is positive in the sense described above.
\end{rmk}

First, we need calculus rules in \RCD\,space from Gigli \cite[Proposition 3.17,4.11]{G}.
\begin{lem}\label{rule}
	Let $\Omega$ be a domain in \RCD\,space for some $K\in\mathbb{R}$ and $N\in[1,\infty)$. Then we have following chain rule and Leibniz rule.\\
	$(i)$ Let $f\in \W(\Omega)\cap L^{\infty}(\Omega)$ and $h\in C^2(\R)$. Then we have
	\begin{equation}\label{cal1}
	\Delta[ h(f)]=h'(f)\cdot \Delta f+h''(f)\cdot |\nabla f|^2.
	\end{equation}
	$(ii)$ Let $f,g\in \W(\Omega)\cap L^{\infty}(\Omega)$. Then we have
	\begin{equation}\label{cal2}
	\Delta (f\cdot g)=\Delta f\cdot g+f\cdot \Delta g+2\left\langle \nabla f,\nabla g\right\rangle.
	\end{equation}
\end{lem}
\begin{rmk}
	Here, \eqref{cal1} and \eqref{cal2} hold in the sense of distribution. If $f,g\in{\rm LIP}(\Omega)\cap\mathcal{D}(\Omega)\cap L^{\infty}(\Omega)$, then above formulas hold as ordinary sense.	
\end{rmk}
Before  proving theorems in Section 1, we need the following fact about $|\nabla u|^2$ of weak solution of equation \eqref{GE} from \cite[Lemma 3.12]{LU2} which is a parabolic version of the following Lemma \ref{reg}.
\begin{lem}\label{reg}
	Let $(X,d,\mu)$ be a \RCD\,space with $K\in\mathbb{R}$ and $N\in[0,\infty)$. Let $u(x)\in \W(B(x_0,2R))\cap L^{\infty}(B(x_0,2R))$ be a positive weak solution of equation \eqref{GE} on $B(x_0,2R)$. Then we have $|\nabla u|^2 \in \W(B(x_0,R))\cap L^{\infty}(B(x_0,R))$.	
\end{lem}

\section{Basic elliptic inequalities and lemmas}\label{s3}

In this section, we present several computational lemmas that are integral to  proving main theorems. The core of the section are constructions of auxiliary functions about solutions.

Let $u\in \W(B(x_0,2R))\cap L^{\infty}(B(x_0,2R))$ be a positive weak solution of \eqref{GE} on $B(x_0,2R)$. First, for $\b\neq 0$, we set
\begin{equation}\label{tran}
w=u^{-\beta},
\end{equation}
then chain and Leibniz rule in Lemma \ref{rule} give the equivalent equation of \eqref{GE}:
\begin{equation}\label{wle}
\Delta w=\Big(1+\frac{1}{\beta}\Big)\frac{|\nabla w|^2}{w}+\b w\frac{f(u)}{u}.
\end{equation}
For undetermined real numbers $\g,\epsilon,d\ge 0$, we define \emph{the first kind} auxiliary function:
\begin{equation}\label{F}
F=(u+\epsilon)^{-\b\g}\left(\frac{|\nabla w|^2}{w^2}+d\frac{f(u)}{u}\right).
\end{equation}
Here the case $\e>0$ is essentially important for our estimates on curvature negative spaces. When $\epsilon=0$ and $\g=1$, $F$ is also called a $P$-function in  \cite{CFP} (see also \cite{CM,LU0, LU3, W}), which is used to obtain Liouville-type results for critical or subcritical equations in the Ricci non-negative manifolds.
 We also need to consider the following transformation 
\begin{equation}\label{tran2}
w=(u+\epsilon)^{-\b},
\end{equation}
where $\e>0$ is undetermined. Then transformation \eqref{tran2} gives the equivalent form of \eqref{GE}:
\begin{equation}\label{wlee}
\d w=\left(1+\frac{1}{\b}\right)\frac{|\nabla w|^2}{w}+\b w^{1+\frac{1}{\b}}f(u).
\end{equation}
We define \emph{the second kind} auxiliary function:
\begin{equation}\label{G}
G=w^{\g}\left(\frac{|\nabla w|^2}{w^{2}}+d\frac{f(u)}{u}\right),
\end{equation}
where $\g, d\ge0$ are undetermined real numbers.

The following two lemmas provide elliptic inequalities for the auxiliary functions $F$ and $G$, respectively.
\begin{lem}\label{k1}
	Let $u$ be a positive solution of equation \eqref{GE} on $B(x_0,2R)$. The function $w$ is defined by equation \eqref{tran}, and $F$ is its corresponding first kind auxiliary function as described in equation \eqref{F}. If $f(u)>0$ when $u>0$, then we have
	\begin{eqnarray}\label{37}
	\Delta F&\ge&-2Kw^{-2}(u+\e)^{-\b\g}|\nabla w|^2+2\left(\frac{1}{\b}-1+\g\frac{u}{u+\e}\right)\left\langle\nabla F,\nabla \ln w\right\rangle\nonumber\\
	&&+(u+\epsilon)^{-\b\g}\left(U\frac{|\nabla w|^4}{w^4}+Vw^{-2}|\nabla w|^2\frac{f(u)}{u}+W\left(\frac{f(u)}{u}\right)^2\right)
	\end{eqnarray}
	on $B(x_0,\frac{3}{2}R)$\footnote{This inequality means that $\d^s F\ge0$ and $\d^{ac}F\ge \text{RHS of} \,\,\eqref{37}$ in the measure-theoretic sense, where $\d^sF$ and $\d^{ac}F$ are singular part and absolutely continuous part of $\d F$ with respect to $\mu$, respectively. All inequalities in the following description are understood to have this meaning.  }, where
	\begin{displaymath}
	\begin{aligned}
	U=&\,\frac{2}{N}\left(1+\frac{1}{\b}\right)^2+\left(\frac{\g}{\b}-\g^2\right)\frac{u^2}{(u+\e)^2}+2\left(1-\frac{1}{\b}\right)\frac{\g u}{u+\e}-2,\\
V=&\,\frac{4}{N}(1+\b)+2\left(1-\frac{uf'}{f}\right)+d\left(\frac{1}{\b^2}\frac{u^2f''}{f}-\frac{2}{\b}\left(\frac{uf'}{f}-1\right)\right)\nonumber\\
&+\frac{\g u}{u+\e}\left(\b+d\left(\left(\frac{1}{\b}-\g\right)\frac{u}{u+\e}+2-\frac{2}{\b}\right)\right),\\
W=&\,\frac{2\b^2}{N}+d\left(\frac{\b\g u}{u+\e}+1-\frac{uf'}{f}\right).\nonumber
	\end{aligned}
	\end{displaymath}

\end{lem}
\begin{proof}
	By calculus rules in Lemma \ref{rule}, we have
	\begin{eqnarray}\label{11}
	\d F&=&\d((u+\e)^{-\b\g})\left(\frac{|\nabla w|^2}{w^2}+d\frac{f(u)}{u}\right)\nonumber\\
	&&+2\left(\left\langle\nabla (u+\e)^{-\b\g},\nabla\left(\frac{|\nabla w|^2}{w^2}\right)\right\rangle+d\left\langle\nabla (u+\e)^{-\b\g},\nabla\left(\frac{f}{u}\right)\right\rangle\right)\nonumber\\
	&&+(u+\e)^{-\b\g}\left(\d\left(\frac{|\nabla w|^2}{w^2}\right)+d\d\left(\frac{f}{u}\right)\right).
	\end{eqnarray}
	We compute terms in \eqref{11} as follows. First, by equation \eqref{GE}, transformation \eqref{tran} and calculus rule again, we  derive
	\begin{equation}\label{12}
	\d((u+\e)^{-\b\g})=\b\g (u+\e)^{-\b\g-1}f(u)+\b\g(\b\g+1)(u+\e)^{-\b\g-2}|\nabla u|^2,
	\end{equation}
	\begin{equation}\label{13}
	\d\left(\frac{f(u)}{u}\right)=\left(\frac{f}{u}-f'\right)\frac{f}{u}+\left(\frac{f''}{u}+\frac{2f}{u^3}-2\frac{2f'}{u^2}\right)|\nabla u|^2,
	\end{equation}
	\begin{equation}\label{14}
	\left\langle\nabla (u+\e)^{-\b\g},\nabla\left(\frac{f}{u}\right)\right\rangle=-\b\g (u+\e)^{-\b\g-1}\left(\frac{f'}{u}-\frac{f}{u^2}\right)|\nabla u|^2,
	\end{equation}
	\begin{eqnarray}\label{15}
	&&\left\langle\nabla (u+\e)^{-\b\g},\nabla\left(\frac{|\nabla w|^2}{w^2}\right)\right\rangle\nonumber\\
	&&=-2\g(u+\e)^{-\b\g-1}u\frac{|\nabla w|^4}{w^4}+\g(u+\e)^{-\b\g-1}w^{-\frac{1}{\b}-3}\left\langle\nabla w,\nabla|\nabla w|^2\right\rangle.
	\end{eqnarray}
	Then we deal with the term $\d\left(\frac{|\nabla w|^2}{w^2}\right)$.
	\begin{eqnarray}\label{16}
	\d\left(\frac{|\nabla w|^2}{w^2}\right)&=&\d(w^{-2})|\nabla w|^2+w^{-2}\d(|\nabla w|^2)+2\left\langle\nabla|\nabla w|^2,\nabla(w^{-2})\right\rangle\nonumber\\
	&=&\left(-2w^{-3}\d w+6w^{-4}|\nabla w|^2\right)|\nabla w|^2-4w^{-3}\left\langle\nabla w,\nabla|\nabla w|^2\right\rangle+w^{-2}\d(|\nabla w|^2)\nonumber\\
	&=&-2w^{-3}|\nabla w|^2\left(\left(1+\frac{1}{\b}\right)\frac{|\nabla w|^2}{w}+\b w\frac{f(u)}{u}\right)+6w^{-4}|\nabla w|^4\nonumber\\
	&&-4w^{-3}\left\langle\nabla w,\nabla|\nabla w|^2\right\rangle+w^{-2}\d(|\nabla w|^2)\nonumber\\
	&\ge&\left(4-\frac{2}{\b}\right)w^{-4}|\nabla w|^4-2\b w^{-2}|\nabla w|^2\frac{f(u)}{u}\nonumber\\
	&&-4w^{-3}\left\langle\nabla w,\nabla|\nabla w|^2\right\rangle+\frac{2}{N}w^{-2}(\d w)^2-2Kw^{-2}|\nabla w|^2\nonumber\\
	&&+2w^{-2}\left\langle\nabla w,\nabla\left(\left(1+\frac{1}{\b}\right)\frac{|\nabla w|^2}{w}+\b w\frac{f}{u}\right)\right\rangle\nonumber\\
	&=&2\left(\frac{1}{\b}-1\right)w^{-3}\left\langle\nabla w,\nabla|\nabla w|^2\right\rangle+\left(2-\frac{4}{\b}\right)w^{-4}|\nabla w|^4\nonumber\\
	&&+2\left(\frac{f}{u}-f'\right)w^{-2}|\nabla w|^2++\frac{2}{N}w^{-2}(\d w)^2-2Kw^{-2}|\nabla w|^2
	\end{eqnarray}
	on $B(x_0,\frac{3}{2}R)$,
	where we use calculus rule, equation \eqref{wle}, Bochner inequality in Lemma \ref{Bochner} and equation \eqref{wle} again for the second equation, the third equation, the fourth inequality and last equation, respectively.  Notice that
	\begin{eqnarray}\label{19}
	(u+\e)^{-\b\g}w^{-3}\left\langle\nabla w,\nabla|\nabla w|^2\right\rangle=\left\langle\nabla F,\nabla\ln w\right\rangle+\left(2-\frac{\g u}{u+\e}\right)(u+\e)^{-\b\g}\cdot\frac{|\nabla w|^4}{w^4}\nonumber\\
	-d\left(\frac{\g u}{u+\e}+\frac{1}{\b}\left(1-\frac{uf'}{f}\right)\right)\cdot\frac{f}{u}\cdot\frac{|\nabla w|^2}{w^2}.
	\end{eqnarray}
	Substitute \eqref{12}-\eqref{16} into \eqref{11} first, then we use \eqref{wle} and \eqref{19} to finish the proof.

\end{proof}
Following a similar process as the proof of Lemma \ref{k1}, we obtain the following elliptic inequality for the second kind auxiliary function $G$.
\begin{lem}\label{k2}
	Let $u$ be a positive solution of \eqref{GE} on $B(x_0,2R)$. The function $w$ is defined by equation \eqref{tran2}, and $G$ is its corresponding second kind auxiliary function as described in equation \eqref{G}. If $f(u)>0$ when $u>0$, then we have
	\begin{eqnarray}\label{k22}
	\Delta G&\ge&-2Kw^{\g-2}|\nabla w|^2+2\left(\frac{1}{\b}-1+\g\right)\left\langle\nabla G,\nabla \ln w\right\rangle\nonumber\\
	&&+Xw^{\g-4}|\nabla w|^4+Yw^{\g-2}|\nabla w|^2\frac{f(u)}{u}+Zw^{\g}\left(\frac{f(u)}{u}\right)^2
	\end{eqnarray}
	on $B(x_0,\frac{3}{2}R)$, where
	\begin{displaymath}
	\begin{aligned}
	X=&\,\frac{2}{N}\left(1+\frac{1}{\b}\right)^2+2\g-\g^2-\frac{\g}{\b}-2,\\
	Y=&\,\left(\frac{4}{N}(1+\b)+2+\g\b\right)\frac{u}{u+\e}-2\frac{uf'}{f}+2d\left(\g-1+\frac{1}{\b}\right)\left(\frac{1}{\b}\frac{u+\e}{u}\left(\frac{uf'}{f}-1\right)-\g\right)\\
	&+d\left(\frac{1}{\b^2}\frac{(u+\e)^2}{u^2}\left(\frac{u^2f''}{f}+2-2\frac{uf'}{f}\right)+\g\left(\g+\frac{1}{\b}\right)-\frac{2\g}{\b}\frac{u+\e}{u}\left(\frac{uf'}{f}-1\right)\right),\\
	Z=&\,\frac{2\b^2}{N}\left(\frac{u}{u+\e}\right)^2+d\left(\frac{\b\g u}{u+\e}+1-\frac{uf'}{f}\right).\nonumber
	\end{aligned}
	\end{displaymath}

\end{lem}
\begin{rmk}\label{lb}
\rm	From the proof of Lemma \ref{k1}, it is clear that the existence of  positive upper and lower bounds for the solutions guarantees the validity of the chain rule, which is essential for the above argument.
\end{rmk}

In order to use these elliptic inequalities through the Bernstein-Yau method (as discussed in Yau \cite{Y} for harmonic functions on Riemannian manifolds), we require the weak maximal principle as stated in Zhang-Zhu \cite{ZZ2}, as well as a cut-off function on \RCD\, metric measure spaces. The following weak maximal principle is a special case of Theorem 4.4 in \cite{ZZ2}.

\begin{lem}\cite[Theorem 4.4]{ZZ2}\label{MMP}
	Let $\Omega$ be a bounded domain. Let $f(x)\in \W(\Omega)$ and suppose that $f$ achieve one of its strict maximum in $\Omega$ in the sense: there exists a neighborhood $U\subset\subset \Omega$ such that
	\begin{equation}
	\sup_{U}f>\sup_{\Omega\setminus U}f.\nonumber
	\end{equation}
	Assume that $\Delta f$ is a signed Radon measure with $\Delta^s f\ge0$. Let $w\in \W(\Omega)$, then for any $\epsilon>0$, we have
	\begin{equation}
	\mu\big(\{
	x\in\O: f(x)>\sup_{U}f-\epsilon \quad\text{and}\quad
	\Delta^{ac} f(x)+\left\langle \nabla f, \nabla w\right\rangle\le \epsilon
	\}\big)>0.\nonumber
	\end{equation}
\end{lem}
Just like in the classical case, the following Laplacian comparison theorem guarantees the existence of a good cut-off function.
\begin{lem}\cite[Remark 5.17]{G}
	Let $(X,d,\mu)$ be an infinitesimally strictly convex $CD^{*}(K,N)$ metric measure space for $K\in\mathbb{R}$ and $N\in(1,\infty)$. For $x_0\in X$, denote by $d_{x_0}: X\to [0,\infty)$ the function $x\mapsto d(x_0,x)$. Then 
	\begin{equation}
	\Delta d_{x_0}|_{X\setminus x_0}\le \frac{N\sigma_{K,N}(d_{x_0})-1}{d_{x_0}}\mu.
	\end{equation}
	Here
	\begin{eqnarray}
	\sigma_{K,N}(\theta)=
	\begin{cases}
	\theta\sqrt{\frac{K}{N}}\cot\big(\theta\sqrt{\frac{K}{N}}\big)\qquad\qquad\quad\, \text{\rm if}\quad K>0\vspace{3mm}\\
	1\qquad \qquad \qquad\qquad\qquad\qquad \text{\rm if}    \quad K=0\vspace{3mm}\\
	\theta\sqrt{\frac{-K}{N}}\coth\big(\theta\sqrt{\frac{-K}{N}}\big)\quad\qquad\;\, \text{\rm if}\quad K<0.
	\end{cases}
	\end{eqnarray}
\end{lem}

\begin{lem}\label{cut off}
{\rm (Existence of cut-off function)}Let $(X,d,\mu)$ be an $\text{RCD}$ space with $K\le 0$ and $N\in[1,\infty)$. Then for any $\alpha\in(0,\frac{1}{2}]$ and $R>0$, there exists a cut-off function $\Phi\in {\rm Lip}(B(x_0,2R))$ such that

\noindent
(i) $\Phi(x)=\phi(d(x_0,x))$, where $\phi\colon [0,\infty)\to\mathbb{R}$ is a non-increasing function satisfying
\begin{eqnarray}
	\phi(t)=
	\begin{cases}
		1 &\qquad t\in [0,R],\\
		\alpha &\qquad t\in \bigl[\frac{5}{4}R,2R\bigr].
	\end{cases} \nonumber
\end{eqnarray}

\noindent
(ii)
\begin{eqnarray}
	\frac{|\nabla \Phi|}{\Phi^{\frac{1}{2}}}\le\frac{C}{R}. \nonumber
\end{eqnarray}

\noindent
(iii)
\begin{eqnarray}
	\Delta \Phi\ge-\frac{C}{R}\sqrt{-NK}\coth\biggl(R\sqrt{\frac{-K}{N}}\biggr)-\frac{C}{R^2} \nonumber
\end{eqnarray}
holds on $B(x_0,2R)$ in the distribution sense, where $C>0$ is a universal constant.
\end{lem}

\section{Proof of Theorem \ref{main}}

In this section, we will present a proof of Theorem \ref{main} by using  Bernstein-Yau method, which serves as  a model for the  proofs that follow. For clarity and uniformity, in this section, we will assume that the nonlinear term $f$ in equation \eqref{GE} satisfies:\\
{\rm (i)} $f(t)>0$ on $(0,\infty)$,\\
{\rm (ii)} $\Lambda<p(n)$,\\
{\rm (iii)}  $\Pi>-\infty$.

Under these assumptions, we have the following lemma by setting $\g=0$ in Lemma \ref{k1}.

\begin{lem}\label{k3}
	Let $u$ be a positive  solution of \eqref{GE} on $B(x_0,2R)$. The function $w$ is defined by \eqref{tran}, and $F$ is its first kind auxiliary function  as described in equation \eqref{F} with $\g=0$. Then there exist $\b=\b(N,\Lambda)>0$ and $d=d(N,\Lambda,\Pi)>0$ such that
	\begin{eqnarray}
	\Delta F&\ge&-2KF+2\left(\frac{1}{\b}-1\right)\left\langle\nabla F,\nabla \ln w\right\rangle+LF^2
	\end{eqnarray}
	on $B(x_0,\frac{3}{2}R)$, where $L=L(N,\Lambda)>0$.
	
\end{lem}

\begin{proof}
	First, by setting $\g=0$ in Lemma \ref{k1}, we have
	\begin{eqnarray}\label{41}
\Delta F&\ge&-2KF+2\left(\frac{1}{\b}-1\right)\left\langle\nabla F,\nabla \ln w\right\rangle\nonumber\\
&&+U\frac{|\nabla w|^4}{w^4}+Vw^{-2}|\nabla w|^2\frac{f(u)}{u}+W\left(\frac{f(u)}{u}\right)^2
\end{eqnarray}
on $B(x_0,\frac{3}{2}R)$, where 
\begin{displaymath}
\begin{aligned}
U=&\,\frac{2}{N}\left(1+\frac{1}{\b}\right)^2-2,\\
V=&\,\frac{4}{N}(1+\b)+2\left(1-\frac{uf'}{f}\right)+d\left(\frac{1}{\b^2}\frac{u^2f''}{f}-\frac{2}{\b}\left(\frac{uf'}{f}-1\right)\right),\\
W=&\,\frac{2\b^2}{N}+d\left(1-\frac{uf'}{f}\right).\nonumber
\end{aligned}
\end{displaymath}
 By our conditions of $f$, for any $\b>0$, we immediately get
\begin{eqnarray}
V&\ge&\,\frac{4}{N}(1+\b)+\left(2+\frac{2d}{\b}\right)\left(1-\Lambda\right)+\frac{d}{\b^2}\Pi,\nonumber\\
W&\ge&\,\frac{2}{N}\b^2+d\left(1-\Lambda\right).\nonumber
\end{eqnarray}	
Therefore, for any $l>0$, if $U\ge l$ and $\frac{2}{N}\b^2+d\left(1-\Lambda\right)\ge l$, we have
\begin{eqnarray}\label{321}
&&	U\frac{|\nabla w|^4}{w^4}+V\frac{f}{u}\frac{|\nabla w|^2}{w^2}+W\left(\frac{f}{u}\right)^2\nonumber\\
&&\ge l\left(\frac{|\nabla w|^4}{w^4}+\left(\frac{f}{u}\right)^2\right)+H(\b,d,l,N,\Lambda,\Pi)\frac{f(u)}{u}\frac{|\nabla w|^2}{w^2},
\end{eqnarray}
 where	
\begin{eqnarray}
H(\b,d,l,N,\Lambda,\Pi)&=&2\sqrt{\left(\frac{2}{N}\left(1+\frac{1}{\b}\right)^2-2-l\right)\left(\frac{2}{N}\b^2+d\left(1-\Lambda\right)-l\right)}\nonumber\\&&+\frac{4}{N}(1+\b)+\left(2+\frac{2d}{\b}\right)\left(1-\Lambda\right)+\frac{d}{\b^2}\Pi.
\end{eqnarray}	
Then 	
\begin{equation}
H(\b,0,0,N,\Lambda,\Pi)=\frac{4}{N}\sqrt{(1+\b)^2-N\b^2}+\frac{4}{N}(1+\b)+2\left(1-\Lambda\right),
\end{equation}
and we choose $\b_0=\frac{2}{N-1}$ for $N>1$, $\b_0=\max(\Lambda,1)$ for $N=1$,	then 
\begin{eqnarray}
H(\b,0,0,N,\Lambda,\Pi)\ge
\begin{cases}
4\qquad\qquad\qquad\qquad \,\,\text{if} \qquad N=1\\
2\left(\frac{N+3}{N-1}-\Lambda\right)\qquad\quad\text{if} \qquad N>1.
\end{cases}
\end{eqnarray}

Therefore, by continuity of $H$, there exist $l_0=l(N,\Lambda)>0$ and $d=d(N,\Lambda,\Pi)\in(0,1)$ such that  

\begin{eqnarray}
H(\b,d_0,l_0,N,\Lambda,\Pi)\ge
\begin{cases}
2\qquad\qquad\qquad\qquad \,\,\text{if} \qquad N=1\\
\frac{N+3}{N-1}-\Lambda\qquad\qquad\quad\text{if} \qquad N>1.
\end{cases}
\end{eqnarray}	
	Then we choose $\b=\b_0$, $d=d_0$ and $l=l_0$, and we finish the proof with $L=\frac{l}{2}$.
		
\end{proof}

Now, we start to prove Theorem \ref{main}.
\begin{proof}[Proof of Theorem \ref{main}] Let $u$ be a positive solution on $B(x_0,2R)$. We fix $\gamma=0$, $\beta=\beta(N,\Lambda)>0$, $d=d(N,\Lambda,\Pi)>0$, and $L=L(N,\Lambda)>0$ such that Lemma~\ref{k3} holds (all notations are the same as in Lemma~\ref{k3}).
	
	Let $A:=\Phi F$ be the auxiliary function, where $\Phi$ is a suitable cut-off function as in Lemma~\ref{cut off} with $\alpha=\frac{M_1}{2M_2}$. Here, $M_1=\sup\limits_{B(x_0,R)} A$ and $M_2=\sup\limits_{B(x_0,\frac{3}{2}R)} A$. Without loss of generality, we may assume that $M_1>0$; otherwise, the proof is complete.
	First, by the Radon--Nikodym decomposition, we obtain
\begin{equation}\label{sing eq}
\Delta^{s}A=\Delta^s\Phi \cdot F+\Phi\cdot \Delta^s F\ge 0
\end{equation}
on $B(x_0,\frac{3R}{2})$ because the singular part of $\d \Phi$ and $\d F$ are non-negative by Lemma \ref{k1} and Lemma \ref{cut off}.
By chain rule and $\Phi\ge\alpha$, we also have
\begin{eqnarray}\label{aeq}
\d^{ac}A&=&\d^{ac}\Phi\cdot F+\d^{ac}F\cdot \Phi+2\left\langle\nabla \Phi,\nabla F\right\rangle\nonumber\\
&=& \Delta^{ac} \Phi\cdot F+\d^{ac}F\cdot \Phi+\frac{2}{\Phi}\left\langle\nabla \Phi,\nabla(\Phi F)-\nabla\Phi\cdot F\right\rangle\nonumber\\
&=&\Delta^{ac} \Phi\cdot F+\d^{ac}F\cdot \Phi+2\left\langle\frac{\nabla \Phi}{\Phi} ,\nabla A \right\rangle-2A\cdot\frac{|\nabla \Phi|^2}{\Phi^2}.
\end{eqnarray}
By Lemma \ref{k3} and chain rule  again, we have
\begin{eqnarray}\label{42}
\d^{ac}A&\ge& -2KA+2\left(\frac{1}{\b}-1\right)\left\langle\nabla A,\nabla \ln w\right\rangle+L\Phi F^2\nonumber\\
&&-2\left(\frac{1}{\b}-1\right)F\left\langle\nabla \Phi,\nabla \ln w\right\rangle+F\Delta^{ac} \Phi  +2\left\langle\frac{\nabla \Phi}{\Phi} ,\nabla A \right\rangle-2A\frac{|\nabla \Phi|^2}{\Phi^2}.
\end{eqnarray}
Cauchy-Schwarz inequality and basic equality give
\begin{eqnarray}\label{bb}
	-2\left(\frac{1}{\b}-1\right)F\left\langle\nabla \Phi,\nabla \ln w\right\rangle&\ge&-\left(\frac{1}{\b}-1\right)^2\frac{2}{L}\frac{|\nabla\Phi|^2}{\Phi}F-\frac{L}{2}\Phi \frac{|\nabla w|^2}{w^2}F\nonumber\\
	&\ge&-\left(\frac{1}{\b}-1\right)^2\frac{2}{L}\frac{|\nabla\Phi|^2}{\Phi}F-\frac{L}{2}\Phi F^2.
\end{eqnarray}
Substituting \eqref{bb} into \eqref{42} yields
\begin{eqnarray}
\d^{ac}A+\left\langle\nabla A,\nabla B\right\rangle&\ge& \Delta^{ac} \Phi\cdot F-2KA+\frac{L}{2}\Phi F^2-2\left(\frac{1}{L}\left(\frac{1}{\b}-1\right)^2+1\right)\frac{|\nabla \Phi|^2}{\Phi^2}A,\nonumber
\end{eqnarray}
where $B=2\left(1-\frac{1}{\b}\right)\ln w-2\ln \Phi$.

Notice that $A(x)$ achieves its strict maximum in $B(x_0,\frac{5}{4}R)$ in the sense of Lemma \ref{MMP}. Therefore, by Lemma \ref{MMP} ($\O=B(x_0,\frac{3}{2}R)$ and $U=B(x_0,\frac{5}{4}R)$), we have a sequence $\{x_j\}_{j=j_0}^{\infty}\subset B(x_0,\frac{3}{2}R)$ such that 
\begin{equation}
	A(x_j)\ge \sup\limits_{B(x_0,\frac{3}{2}R)}A-\frac{1}{j}>0
\end{equation}
and 
\begin{equation}\label{43}
\frac{1}{j}\ge \Delta^{ac} \Phi\cdot F(x_j)-2KA(x_j)+\frac{L}{2}\Phi F^2(x_j)-2\left(\frac{1}{L}\left(\frac{1}{\b}-1\right)^2+1\right)\frac{|\nabla \Phi|^2}{\Phi^2}A(x_j).
\end{equation}
Multiplying $\Phi(x_j)$ on both sides of \eqref{43} gives
\begin{equation}\label{44}
\frac{1}{j}\ge \Delta^{ac} \Phi\cdot A(x_j)-2K\Phi A(x_j)+\frac{L}{2}A^2(x_j)-2\left(\frac{1}{L}\left(\frac{1}{\b}-1\right)^2+1\right)\frac{|\nabla \Phi|^2}{\Phi}A(x_j).
\end{equation}
So, by the property of $\Phi$ in Lemma \ref{cut off}, we have
\begin{equation}\label{45}
\frac{1}{j}\ge \frac{L}{2}A^2(x_j)-2KA(x_j)-C(N,\Lambda)\left(\frac{\sqrt{K}}{R}+\frac{1}{R^2}\right)A(x_j),
\end{equation}
where we use $\sqrt{NK}\coth\left(R\sqrt{\frac{K}{N}}\right)\le C(N)(\sqrt{K}+R^{-1})$. Letting $j\to\infty$ in \eqref{45}, we obtain
\begin{equation}
	\sup\limits_{B(x_0,\frac{3}{2}R)}A\le C(N,\Lambda)\left(K+\frac{1}{R^2}\right),
\end{equation}
which implies
\begin{equation}
\sup\limits_{B(x_0,R)}\left(\frac{|\nabla u|^2}{u^2}+\frac{f(u)}{u}\right)\le C(N,\Lambda,\Pi)\left(K+\frac{1}{R^2}\right).
\end{equation}
We completed the proof.

\end{proof}

\section{Proof of Theorem \ref{main1}}
In this section, we do not require  the nonlinear term $f$ to be positive, as in Section 4. However, the cost of removing this assumption is that we cannot obtain the same bound information for the solutions. To illustrate this, let us reconsider Lemma \ref{k1} and set $\e=\g=d=0$, we can see that the positivity condition of $f$ is not necessary. 
\begin{lem}\label{k4}
	Let $u\in\W(B(x_0,2R))\cap L^{\infty}(B(x_0,2R))$ be a positive solution of \eqref{GE} on $B(x_0,2R)$, $w$ be defined by \eqref{tran} and $F$ be its first kind auxiliary function as \eqref{F} with $\e=\g=d=0$. Then we have
	\begin{eqnarray}\label{51}
	\Delta F&\ge&-2KF+2\left(\frac{1}{\b}-1\right)\left\langle\nabla F,\nabla \ln w\right\rangle+\left(\frac{2}{N}\left(1+\frac{1}{\b}\right)^2-2\right)F^2\nonumber\\
	&&+\left(\left(\frac{4}{N}(1+\b)+2\right)\frac{f(u)}{u}-2f'(u)\right)\frac{|\nabla w|^2}{w^2}+\frac{2\b^2}{N}\left(\frac{f(u)}{u}\right)^2
	\end{eqnarray}
	on $B(x_0,\frac{3}{2}R)$.

\end{lem}

Through careful manipulation of the coefficients in \eqref{51}, a similar process to the proof of Lemma \ref{k3} yields the following result.
\begin{lem}\label{k5}
	Let $u$ be a positive solution of \eqref{GE} on $B(x_0,2R)$. The $w$ is defined by equation \eqref{tran}, and $F$ is its first kind auxiliary function as described in equation \eqref{F} with $\e=\g=d=0$. If there exists $\a\in(1,p(N))$ such that $t^{-\a}f(t)$ is non-increasing on $(0,\infty)$,  then  there exists a $ \b=\b(N,\a)>0$ such that
	\begin{eqnarray}
	\Delta F&\ge&-2KF+2\left(\frac{1}{\b}-1\right)\left\langle\nabla F,\nabla \ln w\right\rangle+LF^2
	\end{eqnarray}
	on $B(x_0,\frac{3}{2}R)$, where $L=L(N,\a)>0$.
\end{lem}
\begin{proof}
	We divide the proof into two cases.
	
	Case 1: $\a\in(1,1+\frac{4}{N}]$.
We can choose $\b=\b(N,\a)=\min(\frac{1}{N},\frac{N}{2}(\a-1))$, by basic inequality, then we have	
	\begin{eqnarray}\label{ba}
	&&	\left(\frac{2}{N}\left(1+\frac{1}{\b}\right)^2-4\right)\frac{|\nabla w|^4}{w^4}+\frac{2\b^2}{N}\left(\frac{f(u)}{u}\right)^2\nonumber\\
	&&\ge\frac{4}{N}\sqrt{(1+\b)^2-2n\b^2}\cdot\left|\frac{f(u)}{u}\right|\cdot\frac{|\nabla w|^2}{w^2}\ge\frac{4}{N}\left|\frac{f(u)}{u}\right|\cdot\frac{|\nabla w|^2}{w^2}.
	\end{eqnarray}	
	Substituting \eqref{ba}	into \eqref{51} gives 
	\begin{eqnarray}
	\d F&\ge& -2KF+2\left(\frac{1}{\b}-1\right)\left\langle\nabla F,\nabla\ln w\right\rangle+2F^2	\nonumber\\
	&&+\left(\left(\frac{4}{N}(1+\b)+2\right)\frac{f(u)}{u}+\frac{4}{N}\left|\frac{f(u)}{u}\right|-2f'(u)\right)\frac{|\nabla w|^2}{w^2}.
	\end{eqnarray}
	One can see
	\begin{equation}
	\left(\frac{4}{N}(1+\b)+2\right)\frac{f(u)}{u}+\frac{4}{N}\left|\frac{f(u)}{u}\right|-2f'(u)\ge 2\a\frac{f(u)}{u}-2f'(u)\ge 0,
	\end{equation}
	which is due to that $t^{-\a}f(t)$ is non-increasing.
	
	Case 2: $\a\in(1+\frac{4}{N},p(N))$. As \eqref{ba}, for any $l\in(2,\infty)$, we have
	\begin{equation}\label{2ba}
	\left(\frac{2}{N}\left(1+\frac{1}{\b}\right)^2-l\right)\frac{|\nabla w|^4}{w^4}+\frac{2\b^2}{N}\left(\frac{f(u)}{u}\right)^2\ge\frac{4}{N}\sqrt{(1+\b)^2-\frac{Nl}{2}\b^2}\cdot\left|\frac{f(u)}{u}\right|\cdot\frac{|\nabla w|^2}{w^2}.
	\end{equation}
	Combining \eqref{2ba} and \eqref{51}, we obtain
	\begin{eqnarray}\label{2bb}
	\d F&\ge& -2KF+2\left(\frac{1}{\b}-1\right)\left\langle\nabla F,\nabla\ln w\right\rangle+(l-2)F^2	\nonumber\\
	&&+\left(\left(\frac{4}{N}(1+\b)+\frac{4}{N}\sqrt{(1+\b)^2-\frac{Nl}{2}\b^2}+2\right)\frac{f(u)}{u}-2f'(u)\right)\frac{|\nabla w|^2}{w^2}.
	\end{eqnarray}
	For any $l\in(2,\infty)$ and $N\ge 1$, we define
	\begin{equation}
	g(\b,N,l)=\frac{4}{N}(1+\b)+\frac{4}{N}\sqrt{(1+\b)^2-\frac{Nl}{2}\b^2}+2.
	\end{equation}
	One can see $g(\cdot,N,l)$ is increasing on $[0,\frac{4}{Nl-2}]$ and $g(0,N,l)\equiv 2+\frac{8}{N}$, $g(\frac{4}{Nl-2},n,l)=2+\frac{8l}{Nl-2}$. Now, we choose $l\in(2,\infty)$ such that $\frac{4l}{Nl-2}= \a-1$ because $\a\in(1+\frac{4}{N},p(N))$. For this fixed $l$ and set $\b=\frac{4}{Nl-2}$, then \eqref{2bb} becomes
	\begin{eqnarray}\label{2bc}
	\d F\ge-2KF+2\left(\frac{1}{\b}-1\right)\left\langle\nabla F,\nabla\ln w\right\rangle
	+(l-2)F^2	+\left(2\a\frac{f(u)}{u}-2f'(u)\right)\frac{|\nabla w|^2}{w^2}.
	\end{eqnarray}
	By the condition again, we completed the proof with $L=l-2$.
	
\end{proof}

Using Lemma  \ref{k5} and following the same argument as in the proof of Theorem \ref{main}, one can prove Theorem \ref{main1}. We will omit the redundant details.

\section{Relations among basic estimates}
In the section, we investigate the relationships among three fundamental estimates of positive solutions for equation \eqref{GE}.  These results have their own significance beyond their application in Section 7, as they may be obtained through other methods such as integral estimate methods.

Throughout this section, we assume that \X\,\,be a \RCDn\,space with $N\in[1,\infty)$ and $K\ge 0$, and let $u$ be a positive solution of equation \eqref{GE} on $B(x_0,R)$. We will explore the following estimates:\\
{\rm (1)} Universal boundedness estimate:
\begin{equation}\label{ube}
\sup\limits_{B(x_0,R)}\frac{f(u)}{u}\le C_U\left(K+\frac{1}{R^2}\right).
\end{equation}
{\rm (2)} Logarithmic gradient estimate:
\begin{equation}\label{log}
\sup\limits_{B(x_0,R)}\frac{|\nabla u|^2}{u^2}\le C_L\left(K+\frac{1}{R^2}\right).
\end{equation}	
{\rm (3)} Harnack inequality:
\begin{equation}\label{cfhk}
\sup\limits_{B(x_0,R)} u\le C_{H}\inf\limits_{B(x_0,R)}u.
\end{equation}
Here, constants $C_U, C_L, C_H$ do not depend on solutions.

For unified description, we provide the following definition.
\begin{definition}
	For any $\kappa\in[1,\infty)$, we say that estimate (I) $\kappa$-implies estimate  (II)  if the validity of  estimate (I) on $B(x_0,\kappa R)$ implies the validity of estimate (II) on $B(x_0,R)$. We also say that estimate (I) and (II) are $\kappa$-equivalent if they  $\kappa$-imply each other.
\end{definition}
\begin{lem}\label{l31}
If there exists a $c\in\mathbb{R}$ such that $t^{-c}f(t)$ is non-increasing and $f(t)\ge0$ when $t>0$, then universal boundedness estimate $2$-implies logarithmic gradient estimate.	Moreover, if \eqref{ube} holds with constant $C_U$, then \eqref{log} holds with $C_L=C_L(N,c,C_U)$.
\end{lem}
\begin{proof}
	We just need to prove that if  $u$ is a positive solution of \eqref{GE} on $B(x_0,2R)$ and satisfies \eqref{ube} on $B(x_0,2R)$, then \eqref{log} holds on $B(x_0,R)$.
	
If we set $v=\ln u$, then $v$ is a weak solution of
\begin{equation}\label{veq}
	\d v+|\nabla v|^2+\frac{f(u)}{u}=0.
\end{equation}
Now, we define $H=|\nabla v|^2$, then Bochner inequality in Lemma \ref{Bochner} gives
\begin{eqnarray}
	\d H&\ge&\frac{2}{N}\left(\d v\right)^2+2\langle\nabla\d v,\nabla v\rangle-2KH\nonumber\\
	&=&\frac{2}{N}\left(\d v\right)^2-2\langle\nabla H,\nabla v\rangle-2\left(f'(u)-\frac{f(u)}{u}\right)H-2KH\nonumber\\
	&\ge&\frac{2}{N}H^2-2\langle\nabla H,\nabla v\rangle-2\left(f'(u)-\left(1+\frac{2}{N}\right)\frac{f(u)}{u}\right)H-2KH
\end{eqnarray}
on $B(x_0,\frac{3}{2}R)$.
Notice that $t^{-c}f(t)$ is non-increasing  implies $f'(u)\le c\frac{f(u)}{u}$. Hence
\begin{equation}\label{61}
	\d H\ge\frac{2}{N}H^2-2\langle\nabla H,\nabla v\rangle+2\left(1+\frac{2}{N}-c\right)\frac{f(u)}{u}H-2KH
\end{equation}
on $B(x_0,\frac{3}{2}R)$. Because $u^{-1}f(u)$ is non-negative on $B(x_0,2R)$ and condition \eqref{ube} holds, we can rewrite \eqref{61} as
\begin{equation}\label{62}
\d H\ge\frac{2}{N}H^2-2\langle\nabla H,\nabla v\rangle-C(N,c,C_U)\left(K+R^{-2}\right)H-2KH
\end{equation}
on $B(x_0,\frac{3}{2}R)$. Then based on \eqref{62}, we use similar argument as  proof of Theorem \ref{main} to get logarithmic gradient estimate. For completeness, we give its details as follows.

Let $A:=\Phi H$ be the auxiliary function, where $\Phi$ is an undetermined cut-off function as Lemma \ref{cut off}
by choosing $\alpha=\frac{M_1}{2M_2}$, where $M_1=\sup\limits_{B(x_0,R)} A$ and $M_2=\sup\limits_{B(x_0,\frac{3}{2}R)} A$. Without loss of generality, we may assmue that $M_1>0$ or proof is finished. By Radon--Nikodym decomposition, we get 
\begin{equation}\label{6sing eq}
\Delta^{s}A=\Delta^s\Phi \cdot H+\Phi\cdot \Delta^s H\ge 0
\end{equation}
on $B(x_0,\frac{3R}{2})$ because the singular part of $\d \Phi$ and $\d H$ are non-negative by Lemma \ref{cut off} and \eqref{62}.
By chain rule and $\Phi\ge\alpha$, as before, we  have
\begin{eqnarray}\label{6aeq}
\d^{ac}A=\Delta^{ac} \Phi\cdot H+\d^{ac}H\cdot \Phi+2\left\langle\frac{\nabla \Phi}{\Phi} ,\nabla A \right\rangle-2A\cdot\frac{|\nabla \Phi|^2}{\Phi^2}.
\end{eqnarray}
Substituting \eqref{62} into \eqref{6aeq}, we get
\begin{eqnarray}\label{642}
\d^{ac}A&\ge& -2KA-2\left\langle\nabla A,\nabla v\right\rangle+\frac{2}{N}\Phi H^2-C(N,c,C_U)\left(K+R^{-2}\right)A\nonumber\\
&&+2H\left\langle\nabla \Phi,\nabla v\right\rangle+\Delta^{ac} \Phi\cdot H+2\left\langle\frac{\nabla \Phi}{\Phi} ,\nabla A \right\rangle-2A\cdot\frac{|\nabla \Phi|^2}{\Phi^2}.
\end{eqnarray}
Cauchy-Schwarz inequality and basic equality give
\begin{eqnarray}\label{6bb}
2H\left\langle\nabla \Phi,\nabla v\right\rangle\ge-\frac{1}{N}\Phi H^2-N\frac{|\nabla \Phi|^2}{\Phi}H.
\end{eqnarray}
Substituting \eqref{6bb} into \eqref{642} yields
\begin{eqnarray}
\d^{ac}A+\left\langle\nabla A,\nabla B\right\rangle&\ge& \Delta^{ac} \Phi\cdot H-2KA+\frac{1}{N}\Phi H^2\nonumber\\
&&-2\left(N+1\right)\frac{|\nabla \Phi|^2}{\Phi^2}A-C(N,c,C_U)\left(K+R^{-2}\right)A,\nonumber
\end{eqnarray}
where $B=2v-2\ln \Phi$.

Notice that $A$ achieves its strict maximum in $B(x_0,\frac{5}{4}R)$ in the sense of Lemma \ref{MMP}. Therefore, by Lemma \ref{MMP}, we have a sequence $\{x_j\}_{j=j_0}^{\infty}\subset B(x_0,\frac{3}{2}R)$ such that 
\begin{equation}
A_j=A(x_j)\ge \sup\limits_{B(x_0,\frac{3}{2}R)}A-\frac{1}{j}>0
\end{equation}
and 
\begin{eqnarray}\label{643}
\frac{1}{j}&\ge& \Delta^{ac} \Phi\cdot H(x_j)-2KA(x_j)+\frac{1}{N}\Phi H^2(x_j)\nonumber\\
&&-2\left(N+1\right)\frac{|\nabla \Phi|^2}{\Phi^2}A(x_j)-C(N,c,C_U)\left(K+R^{-2}\right)A(x_j).
\end{eqnarray}
Multiplying $\Phi(x_j)$ on both sides of \eqref{643} gives
\begin{eqnarray}\label{644}
\frac{1}{j}&\ge& \Delta^{ac} \Phi\cdot A(x_j)-2K\Phi A(x_j)+\frac{1}{N}A^2(x_j)\nonumber\\
&&-2\left(N+1\right)\frac{|\nabla \Phi|^2}{\Phi}A(x_j)-C(N,c,C_U)\left(K+R^{-2}\right)\Phi A(x_j).
\end{eqnarray}
Hence, by the property of $\Phi$ as before, we have
\begin{equation}\label{645}
\frac{1}{j}\ge \frac{1}{N}A^2_j-2KA_j-C(N)\left(\frac{\sqrt{K}}{R}+\frac{1}{R^2}\right)A_j-C(N,c,C_U)\left(K+R^{-2}\right) A_j.
\end{equation}
 Letting $j\to\infty$ in \eqref{645}, we obtain
\begin{equation}
\sup\limits_{B(x_0,\frac{3}{2}R)}A\le C(N,c,C_U)\left(K+\frac{1}{R^2}\right),
\end{equation}
which implies
\begin{equation}
\sup\limits_{B(x_0,R)}\frac{|\nabla u|^2}{u^2}\le C(N,c,C_U)\left(K+\frac{1}{R^2}\right).
\end{equation}
We completed the proof.

\end{proof}
\begin{rmk}
\rm	Notice that if $c\in (1,p(N))$, by Theorem \ref{main1}, we automatically have the logarithmic gradient estimate \eqref{log} without the condition "$f(u)\ge0$ when $u>0$". When $c\ge p(N)$, from \eqref{61}, we also obtain \eqref{62} by the condition \eqref{ube}, and the following proof does not require the condition $f(u)\ge 0$. Therefore,  if $c>1$, we can remove the condition "$f(u)\ge0$ when $u>0$".
\end{rmk}

To our knowledge, the logarithmic gradient estimate on $B(x_0,R)$ cannot directly derive the Harnack inequality of positive weak solutions, as it may not hold along every geodesic. Therefore, we add a continuity assumption on the solutions to ensure the Harnack inequality. Here, we adopt Garofalo--Mondino's argument in \cite[Proof of Theorem 1.4]{GM}.

\begin{lem}\label{logtohk}
If the positive weak solution of \eqref{GE} is continuous,	then logarithmic gradient estimate $1$-implies Harnack inequality. Moreover, if \eqref{log} holds with constant $C_L$, then \eqref{cfhk} holds with $C_H=e^{2\sqrt{C_L(KR^2+1)}}$.
\end{lem}
\begin{proof}
 For fixed $x\in B(x_0,R)$ and $r\in(0,R-d(x_0,x))$, we define a map $\rho:AC^2([0,1],X)\times [0,1]\to X$ as
\begin{equation}
\rho(l,\tau)=l(\tau),
\end{equation} 
and the function $\kappa:[0,1]\rightarrow\mathbb{R}$ as
\begin{equation}
\kappa(\tau)= \ln u(\rho(l,\tau)).
\end{equation}
Meanwhile, we define (recall that we assume $supp(\mu)=X$)
\begin{equation}
\mu_0^r=\mu_0(B(x_0,r))^{-1}\cdot\mu_0|_{B(x_0,r)}\quad\text{and}\quad 	\mu_1^r=\mu_1(B(x,r))^{-1}\cdot\mu_1|_{B(x,r)}.
\end{equation} 

Let $\P^r\in OptGeo(\mu_0^r,\mu_1^r)$, then the logarithmic gradient estimate in \eqref{log} implies
\begin{equation}\label{kk}
\frac{|\nabla u|^2}{u^2}(\rho(l,\tau))\le C_L\left(K+\frac{1}{R^2}\right)
\end{equation}
holds for $\P^r$-a.e. $l$ and each $\tau\in[0,1]$.

By chain rule and \eqref{kk}, we get
\begin{eqnarray}\label{kk2}
&&\int\kappa(1)-\kappa(0)d\P^r=\int\int_{0}^{1}\kappa'(\tau)d\tau d\P^r\nonumber\\
&\le& \int\int_{0}^{1}|\nabla\ln u|\cdot |\dot{l}|d\tau d\P^r\nonumber\\
&\le&\sqrt{C_L\left(K+\frac{1}{R^2}\right)}\int d(l(0),l(1))d\P^r.
\end{eqnarray}
Notice that  $\P^r$ is concentrated along geodesic connecting points of $B(x_0,r)$ and points of $B(x,r)$ (see \cite{GM}) and continuity of solution, we get
\begin{equation}
\lim\limits_{r\to 0}\int\kappa(1)-\kappa(0)d\P^r
=\ln\left(\frac{u(x)}{u(x_0)}\right)
\end{equation}
and 
\begin{equation}
\lim\limits_{r\to 0}\text{RHS of \eqref{kk2}}= \sqrt{C_L\left(K+\frac{1}{R^2}\right)}\int d(x_0,x)\le\sqrt{C_L\left(KR^2+1\right)}.
\end{equation}
That is, $u(x)\le e^{\sqrt{C_L\left(KR^2+1\right)}}u(x_0)$. Same inequality holds by $x_0,y$ substituting $x,x_0$,
and hence we derive the desired Harnack inequality \eqref{cfhk}.

\end{proof}

\begin{lem}\label{l61}
	We assume that the nonlinear term $f$ satisfies:\\
	{\rm (1)} $f(t)>0$ on $(0,\infty)$,\\
	{\rm (2)} $\Lambda<p_S(N)$,\\
	{\rm (3)} $\frac{tf'}{f}$ is non-decreasing on $(0,\infty)$.\\
	Then  Harnack inequality \eqref{cfhk} $2$-implies universal boundedness estimate \eqref{ube}. Moreover, if \eqref{cfhk} holds with constant $C_{H}$, then \eqref{ube} holds with $C_U=C(N,\Lambda,C_{H})$.	
	
\end{lem}
\begin{proof}
	To prove that estimate \eqref{ube} holds on $B(x_0,R)$, under the assumption that positive solutions obey the Harnack inequality on $B(x_0,2R)$ and the nonlinear term satisfies $(1)$-$(3)$ in Lemma \ref{l61}, we proceed as follows.
	
	Define $w$ as \eqref{tran} and its first kind auxiliary function $F$ as \eqref{F} with $\e=0, \g=1$ and some undetermined $\b>0$, $d>0$, then Lemma \ref{k1} gives
	
	\begin{eqnarray}
\Delta F\ge-2KF+\frac{2}{\b}\left\langle\nabla F,\nabla \ln w\right\rangle+u^{-\b}\left(U\frac{|\nabla w|^4}{w^4}+V\frac{|\nabla w|^2}{w^2}\frac{f(u)}{u}+W\left(\frac{f(u)}{u}\right)^2\right)
\end{eqnarray}
on $B(x_0,\frac{3}{2}R)$, where
\begin{displaymath}
\begin{aligned}
U=&\,\left(\frac{2}{N}\left(1+\frac{1}{\b}\right)^2-1\right)\left(\frac{1}{\b}+1\right),\\
V=&\,\frac{4}{N}(1+\b)+\b+2\left(1-\frac{uf'}{f}\right)+d\left(\frac{1}{\b^2}\frac{u^2f''}{f}-\frac{2}{\b}\frac{uf'}{f}+1+\frac{1}{\b}\right),\\
W=&\,\frac{2\b^2}{N}+d\left(\b+1-\frac{uf'}{f}\right).\nonumber
\end{aligned}
\end{displaymath}	
Now, we claim that there exist $\b=\b(N,\Lambda)>0$, $d=d(N,\Lambda)>0$ and $L=L(N,\Lambda)>0$ such that 	
\begin{eqnarray}\label{claim}
\Delta F\ge-2KF+\frac{2}{\b}\left\langle\nabla F,\nabla \ln w\right\rangle+Lu^{\b}F^2
\end{eqnarray}
on $B(x_0,\frac{3}{2}R)$.

For demonstrating the claim, we divide the proof into three cases as follows.

First, we see that 	
\begin{equation}
	\left(\frac{tf'(t)}{f}\right)'\ge0
\end{equation}
is equivalent to	
\begin{equation}
	\frac{t^2f''}{f}\ge\frac{tf'}{f}\left(\frac{tf'}{f}-1\right),
\end{equation}
which implies 	
\begin{eqnarray}
	V\ge\frac{4}{N}(1+\b)+\b+2\left(1-\frac{uf'}{f}\right)+d\left(\frac{1}{\b}\frac{uf'}{f}-1\right)\left(\frac{1}{\b}\left(\frac{uf'}{f}-1\right)-1\right).
\end{eqnarray}

Case 1: $N\in[1,2]$. We choose $\b=\max(1,\Lambda)$ and $d=d(N,\Lambda)\in(0,1)$ such that 	$V\ge 1$ and $W\ge\frac{2}{N}$. Then we just set $L=\frac{1}{2}\min(U,V,W)$ and we get the desired result.

Case 2: $N\in(2,\infty)$ and $\Lambda<\frac{N+1}{N-2}$. We choose $\b=\b(N,\Lambda)\in(0,\frac{2}{N-2})$ such that 
\begin{eqnarray}
	V_0&=&\frac{4}{N}(1+\b)+\b+2\left(1-\Lambda\right)>0.
\end{eqnarray}
For this fixed $\b$, we choose $d=d(N,\Lambda)\in(0,1)$ such that	$V>\frac{V_0}{2}$ and $W\ge\frac{\b^2}{N}$.  Then we also set $L=\frac{1}{2}\min(U,V,W)$ and we get the desired result. 
	
Case 3: $N\in(2,\infty)$ and $\Lambda\in[\frac{N+1}{N-2},p_S(N))$. 
For obtaining the uniform lower bound of $V$, for any $\b\in(0,\frac{2}{N-2})$ and $d\in(0,\frac{2\b^2}{N(\Lambda-1-\b)}]$, we define the following quadratic polynomial	
\begin{eqnarray}
Q(x,\b,d)=\frac{4}{N}(1+\b)+\b+2\left(1-x\right)+d\left(\frac{x}{\b}-1\right)\left(\frac{x-1}{\b}-1\right).
\end{eqnarray}	
Notice that 
\begin{equation}
\text{the symmetry axis of}\,\,\, Q=\frac{1}{2}+\b+\frac{\b^2}{d}\ge \Lambda,
\end{equation}
where we use $\Lambda\in[\frac{N+1}{N-2},p_S(N))$. Therefore,
\begin{equation}
V\ge Q\left(\frac{uf'}{f},\b,d\right)\ge Q(\Lambda,\b,d)=\frac{4}{n}(1+\b)+\b+2\left(1-\Lambda\right)+\frac{d}{\b^2}\left(\Lambda-\b\right)\left(\Lambda-1-\b\right).
\end{equation}
If we choose $d_0=\frac{2\b^2}{N(\Lambda-1-\b)}$, then there exists $\b=\b(N,\Lambda)\in(0,\frac{2}{N-2})$ such that $Q(\Lambda,\b,d_0)\ge C_0>0$, where $C_0=C_0(N,\Lambda)$. For this fixed $\b$, we can choose $d=d(N,\Lambda)\in (0,d_0)$ such that $V\ge Q(\Lambda,\b,d)\ge \frac{C_0}{2}>0$ and $W>0$. Hence we completed the proof of claim.
	
	By the claim, we can obtain universal boundedness estimate as before.

Let $A:=\Phi F$ be the auxiliary function, where $\Phi$ is an undetermined cut-off function as in Lemma \ref{cut off}
by choosing $\alpha=\frac{M_1}{2M_2}$, where $M_1=\sup\limits_{B(x_0,R)} A$ ($>0$ or proof is finished) and $M_2=\sup\limits_{B(x_0,\frac{3}{2}R)} A$. First, by Radon--Nikodym decomposition, we get 
\begin{equation}\label{6s1}
\Delta^{s}A=\Delta^s\Phi \cdot F+\Phi\cdot \Delta^s F\ge 0
\end{equation}
on $B(x_0,\frac{3R}{2})$ because the singular part of $\d \Phi$ and $\d F$ are non-negative by Lemma \ref{cut off} and \eqref{claim}.
By chain rule and $\Phi\ge\alpha$,  we  have	
\begin{eqnarray}\label{622}
\d^{ac}A=\Delta^{ac} \Phi\cdot F+\d^{ac}F\cdot \Phi+2\left\langle\frac{\nabla \Phi}{\Phi} ,\nabla A \right\rangle-2A\cdot\frac{|\nabla \Phi|^2}{\Phi^2}.
\end{eqnarray}
Substituting \eqref{claim} into \eqref{622}, we obtain
\begin{eqnarray}\label{632}
\d^{ac}A&\ge& -2KA+\frac{2}{\b}\left\langle\nabla A,\nabla \ln w\right\rangle+Lu^{\b}\Phi F^2\nonumber\\
&&-\frac{2}{\b}F\left\langle\nabla \Phi,\nabla \ln w\right\rangle+\Delta^{ac} \Phi\cdot F+2\left\langle\frac{\nabla \Phi}{\Phi} ,\nabla A \right\rangle-2A\cdot\frac{|\nabla \Phi|^2}{\Phi^2}.
\end{eqnarray}	
Cauchy-Schwarz inequality and basic inequality again give	
	\begin{eqnarray}\label{csb}
		-\frac{2}{\b}F\left\langle\nabla \Phi,\nabla \ln w\right\rangle&\ge&-\frac{L}{2}\Phi F\frac{|\nabla w|^2}{w^2}-\frac{2}{\b^2 L}F\frac{|\nabla \Phi|^2}{\Phi}\nonumber\\
		&\ge&-\frac{L}{2}u^{\b}\Phi F^2-\frac{2}{\b^2 L}A\frac{|\nabla \Phi|^2}{\Phi^2}.
	\end{eqnarray}
	Combining \eqref{632} with \eqref{csb}, we have
\begin{equation}
	\d^{ac}A+\left\langle\nabla A,\nabla B\right\rangle\ge \frac{L}{2}u^{\b}\Phi F^2-2KA+\Delta^{ac} \Phi\cdot F-2\left(1+\frac{1}{\b^2 L}\right)A\cdot\frac{|\nabla \Phi|^2}{\Phi^2},
\end{equation}
	where $B=-2\ln \Phi-\frac{2}{\b}\ln w$.
	
Notice that $A$ achieves its strict maximum in $B(x_0,\frac{5}{4}R)$ in the sense of Lemma \ref{MMP}. Therefore, by Lemma \ref{MMP}, we have a sequence $\{x_j\}_{j=j_0}^{\infty}\subset B(x_0,\frac{3}{2}R)$ such that $u(x_j)\ge\inf\limits_{B(x_0,2R)}u>0$,
\begin{equation}
A(x_j)\ge \sup\limits_{B(x_0,\frac{3}{2}R)}A-\frac{1}{j}>0
\end{equation}
and 
\begin{eqnarray}\label{633}
\frac{1}{j}\ge\frac{L}{2}u^{\b}\Phi F^2(x_j)+\Delta^{ac} \Phi\cdot F(x_j)-2KA(x_j)-2\left(1+\frac{1}{\b^2 L}\right)\frac{|\nabla \Phi|^2}{\Phi^2}A(x_j).
\end{eqnarray}
Multiplying $\Phi(x_j)$ on both sides of \eqref{633} yields
\begin{eqnarray}\label{634}
\frac{1}{j}\ge\frac{L}{2}u^{\b} A^2(x_j)+ \Delta^{ac} \Phi\cdot A(x_j)-2K\Phi A(x_j)-2\left(1+\frac{1}{\b^2 L}\right)\frac{|\nabla \Phi|^2}{\Phi}A(x_j).
\end{eqnarray}	
  Using property of $\Phi$ and letting $j\to \infty$, we see	
	\begin{equation}
	\frac{L}{2}\left(\inf\limits_{B(x_0,2R)}u\right)^{\b}\sup\limits_{B(x_0,\frac{3}{2}R)}A\le C(N)\left(K+\frac{1}{R^2}\right),
	\end{equation}
	which implies
		\begin{equation}\label{636}
	\sup\limits_{B(x_0,R)}u^{-\b}\left(\frac{|\nabla u|^2}{u^2}+\frac{f(u)}{u}\right)\le C(N,\Lambda)\left(\inf\limits_{B(x_0,2R)}u\right)^{-\b}\left(K+\frac{1}{R^2}\right).
	\end{equation}
	Then combining Harnack inequality on $B(x_0,2R)$ and \eqref{636} derives 
	\begin{equation}\label{637}
	\sup\limits_{B(x_0,R)}\left(\frac{|\nabla u|^2}{u^2}+\frac{f(u)}{u}\right)\le C(N,\Lambda)C_{H}^{\b}\left(K+\frac{1}{R^2}\right).
	\end{equation}
	We completed the proof.

\end{proof}
\begin{rmk}
	\rm (I) If one carefully modifies  the property of cut-off function, one can find that  universal boundedness estimate $k$-implies logarithmic gradient estimate and Harnack inequality $k$-implies universal boundedness estimate for any $\kappa>1$ under same conditions of Lemma \ref{l31} and \ref{l61}, respectively.\\
	(II) If we divide Harnack inequality into two cases:\\
	(1) Curvature-dependent Harnack inequality:
	\begin{equation}\label{cdh}
	\sup\limits_{B(x_0,R)} u\le e^{C\left(\sqrt{K}R+1\right)}\cdot\inf\limits_{B(x_0,R)}u.
	\end{equation}
	(2) Curvature-free Harnack inequality:
	\begin{equation}\label{cdh}
	\sup\limits_{B(x_0,R)} u\le C\inf\limits_{B(x_0,R)}u.
	\end{equation}
	Here $C$ only depends on $N$ and the nonlinear term $f$.

 The findings of this section can be summarized in the following diagram:\\

	\tikzstyle{startstop} = [rectangle, minimum width=2cm, minimum height=1cm, text centered, draw=black]
	\tikzstyle{point} = [coordinate,on grid,]
	\tikzstyle{arrow} = [thick,->,>=stealth]
	\begin{tikzpicture}[node distance=3cm]
	\node (UB)[startstop]	{Universal boundedness estimate};
	\node (CFH)[startstop,right of=UB, xshift=6cm]  {Curvature-free Harnack};
	\node (LG) [startstop,below of=UB] {Logarithmic gradient estimate};
	\node (CDH) [startstop, right of=LG, xshift=6cm] {Curvature-dependent Harnack};
	\node [point, right of=CDH, node distance=0.2cm] (point1){};
	\node [point, above of=point1, node distance=0.5cm] (point2){};
	\node [point, right of=CFH, node distance=0.2cm] (point3){};	
	\node [point, below of=point3, node distance=0.5cm] (point4){};	
	\draw [arrow](UB)--node[left]{Bernstein}(LG);
	\draw [arrow](LG)--node[above]{Continuity}(CDH);
	\draw [arrow](CFH)--node[above]{Bernstein}(UB);
	\draw [arrow](CFH)--(CDH);
	\draw [arrow](point2)--node[right]{$K=0$}(point4);	
	\end{tikzpicture}
\end{rmk}

\section{Proof of Theorem \ref{main2} and  Theorem \ref{S}}

In this section, we will prove Theorem \ref{main2} and Theorem \ref{S} using the first kind auxiliary function for $N\ge 4$ and  second auxiliary function for $N\in[1,4)$ respectively. When $N\in[1,4)$, by proving the universal boundedness estimate, we can directly obtain the logarithmic gradient estimate through Lemma \ref{l31}. The key to our proof lies in combining truncation method and Bernstein-Yau method. In addition, the role of $\e$ in the constructions of $F,G$ will become clear. 

For unified description, throughout this section, we assume the nonlinear term $f$ satisfies:

\noindent{\rm (1)} $f(t)>0$ on $(0,\infty)$,\\
{\rm (2)} $\lambda\ge 1$,\\
{\rm (3)} $\Lambda<p_S(n)$,\\
{\rm (4)} $\frac{tf'}{f}$ is non-decreasing on $(0,\infty)$.

As the proof of Lemma \ref{l61}, we notice that the condition $(4)$ is equivalent to 
\begin{equation}\label{7k}
\frac{t^2f''}{f}\ge \frac{tf'}{f}\left(\frac{tf'}{f}-1\right).
\end{equation}
Hence above (1), (3), (4) are stronger than the conditions stated in Theorem \ref{main}. Without loss of generality, we assume $\Lambda\ge p(N)$ throughout this section.

The following lemma provides a basic estimate for the coefficients in Lemma \ref{k1}, which plays a crucial role in proving Theorem \ref{main2}.

\begin{lem}\label{ab}
	Let $N\ge 4$, $\e>0, \g=1$, and $U, V, W$ be defined as in Lemma \ref{k1}. Then there exist constants $\b=\b(N,\Lambda)\in(0,\frac{2}{N-2})$, $d=d(N,\Lambda)>0$ and $L=L(N,\Lambda)>0$ such that for any $\e>0$,
	\begin{equation}
	U\ge U_0>0,\nonumber
	\end{equation}
	\begin{equation}
	V\ge V_0-L\chi_{\{u< L\epsilon\}},\nonumber
	\end{equation}
	\begin{equation}
	W\ge W_0-L\chi_{\{u< L\epsilon\}},\nonumber
	\end{equation}
	where $U_0,V_0,W_0$ are positive constants which only depend on $N, \Lambda$.	
\end{lem}
\begin{proof}
	First,
	if $\g=1$, then	
	\begin{eqnarray}
	U&=&\left(\frac{1}{\b}-1\right)\left(\frac{u}{u+\e}-2\right)\frac{u}{u+\e}+\frac{2}{N}\left(1+\frac{1}{\b}\right)^2-2\nonumber\\
	&\ge&\left(\frac{2}{N}\left(1+\frac{1}{\b}\right)-1\right)\left(1+\frac{1}{\b}\right)\nonumber\\
	&>&0,
	\end{eqnarray}
	when $\frac{1}{\b}-1\ge0$ and $\b<\frac{2}{N-2}$, which is equivalent to $\b<\frac{2}{N-2}$ by $N\ge 4$. We define
	\begin{equation}
	U_0=\left(\frac{2}{N}\left(1+\frac{1}{\b}\right)-1\right)\left(1+\frac{1}{\b}\right).\nonumber
	\end{equation}
	At this moment,  
	\begin{eqnarray}
	V&\ge&\frac{4}{N}(1+\b)+2\left(1-\frac{uf'}{f}\right)+d\left(\frac{1}{\b^2}\frac{uf'}{f}-\frac{2}{\b}\right)\left(\frac{uf'}{f}-1\right)\nonumber\\
	&&+\frac{u}{u+\e}\left(\b+d\left(\left(\frac{1}{\b}-1\right)\frac{u}{u+\e}+2-\frac{2}{\b}\right)\right)\nonumber\\
	&=&V(f)+d\left(\frac{1}{\b}-1\right)\frac{\e^2}{(u+\e)^2}-\frac{\b\e}{u+\e},\nonumber\\
	W&=&W(f)-\frac{d\b\e}{u+\e},\nonumber
	\end{eqnarray}
	where
	\begin{eqnarray}
	V(f)&=&\frac{4}{N}(1+\b)+2\left(1-\frac{uf'}{f}\right)+\b+d\left(1+\frac{1}{\b}\left(1-\frac{uf'}{f}\right)\right)\left(1-\frac{1}{\b}\frac{uf'}{f}\right),\nonumber\\
	W(f)&=&\frac{2}{N}\b^2+d\left(\b+1-\frac{uf'}{f}\right).\nonumber
	\end{eqnarray}
	We choose $d_0=\frac{2\b^2}{N(\Lambda-1-\b)}>0$ ($\Lambda\ge\frac{N+3}{N-1}>\frac{N}{N-2}$ when $N\ge 4$), where $\b\in(0,\frac{2}{N-2})$ is undetermined. Define 
	\begin{equation}
	Q(x,\b,d)=\frac{4}{N}(1+\b)+2\left(1-x\right)+\b+d\left(1+\frac{1-x}{\b}\right)\left(1-\frac{x}{\b}\right),
	\end{equation}
	where $d\in(0,d_0]$ and $\b\in(0,\frac{2}{N-2})$ is arbitrary, and notice that 
	\begin{eqnarray}
	\text{the symmetry axis of}\,\,\, Q&=&\frac{1}{2}+\b+\frac{\b^2}{d}\nonumber\\
	&\ge&\b+\frac{1}{2}+\frac{N}{2}(\Lambda-1-\b)\ge\Lambda,
	\end{eqnarray}
	when $\b<\Lambda-\frac{N-1}{N-2}$. Hence, $V(f)\ge Q(\Lambda,\b,d)$. If we choose $d=d_0$ first, then for any $\b\in\left(2\left(\frac{N-1}{N+2}\Lambda-1\right),\Lambda-\frac{N-1}{N-2}\right)$, we have $Q(\Lambda,\b,d_0)>0$. Then we can choose a fixed $\b\in\left(2\left(\frac{N-1}{N+2}\Lambda-1\right),\frac{2}{N-2}\right)$, for this fixed $\b$, by continuity of $Q$, we can fix a $d=d(N,\Lambda)\in(0,d_0)$ such that $Q(\Lambda,\b,d)>0$. It is obvious that for these fixed $\b,d$, $U_0,V(f),W(f)$ both have positive lower bound and the bound only depend on $N,\Lambda$. Therefore, there exist $L=L(N,\Lambda)>0$, $V_0=V_0(N,\Lambda)>0$ and $W_0=W_0(N,\Lambda)>0$ such that
	\begin{eqnarray}
	V&\ge&V_0-L\chi_{\{u< L\epsilon\}},\nonumber\\
	W&\ge&W_0-L\chi_{\{u< L\epsilon\}}.\nonumber
	\end{eqnarray}
	Then the proof is complete.	
	
\end{proof}
\begin{rmk}\label{7r1}
	\rm	Actually, we can choose any $\b\in\left(2\left(\frac{N-1}{N+2}\Lambda-1\right),\frac{2}{N-2}\right)$ and $d\in\left(0,\frac{2\b^2}{N(\Lambda-1-\b)}\right)$ in the proof of Lemma \ref{ab}.
\end{rmk}
With the help of Lemma \ref{k1} and Lemma \ref{ab}, we can give a proof of Theorem \ref{main2} for $N\ge 4$.

\begin{proof}[Proof of Theorem~\ref{main2} in the case $N \ge 4$]
	
	By Lemma \ref{k1} and Lemma \ref{ab}, for $\g=1$ in the construction of $F$, we can choose positive numbers $\b=\b(N,\Lambda)$, $d=d(N,\Lambda)$, $L=L(N,\Lambda)$ and $C=C(N,\Lambda)$ such that for any $\e>0$,
	\begin{equation}\label{70}
	\d F\ge-2KF+2\left(\frac{1}{\b}-1+\frac{u}{u+\e}\right)\left\langle\nabla F,\nabla\ln w\right\rangle+C(u+\e)^{\b}F^2-L\frac{f}{u}\chi_{\{u< L\epsilon\}}F
	\end{equation}
on $B(x_0,\frac{3}{2}R)$.	Notice that $t^{-1}f(t)$ is non-decreasing on $(0,\infty)$ (by $\lambda\ge1$), hence \eqref{70} yields
	\begin{equation}\label{71}
	\d F\ge-2KF+2\left(\frac{1}{\b}-1+\frac{u}{u+\e}\right)\left\langle\nabla F,\nabla\ln w\right\rangle+C(u+\e)^{\b}F^2-\frac{f(L\e)}{\e}F
	\end{equation}
	on $B(x_0,\frac{3}{2}R)$.
	
	As before, we can obtain the boundedness of $F$ via \eqref{71} as follows.
	
	Let $A:=\Phi F$ be the auxiliary function, where $\Phi$ is an undetermined cut-off function as Lemma \ref{cut off}
	by choosing $\alpha=\frac{M_1}{2M_2}$, where $M_1=\sup\limits_{B(x_0,R)} A>0$ (or proof is finished) and $M_2=\sup\limits_{B(x_0,\frac{3}{2}R)} A$. By Radon--Nikodym decomposition, we get 
	\begin{equation}\label{7s}
	\Delta^{s}A=\Delta^s\Phi \cdot F+\Phi\cdot \Delta^s F\ge 0
	\end{equation}
	on $B(x_0,\frac{3R}{2})$ because the singular part of $\d \Phi$ and $\d F$ are non-negative by Lemma \ref{cut off} and \eqref{71}.
	By chain rule and $\Phi\ge\alpha$, we  have	
	\begin{eqnarray}\label{72}
	\d^{ac}A=\Delta^{ac} \Phi\cdot F+\d^{ac}F\cdot \Phi+2\left\langle\frac{\nabla \Phi}{\Phi} ,\nabla A \right\rangle-2A\cdot\frac{|\nabla \Phi|^2}{\Phi^2}.
	\end{eqnarray}
	Substituting \eqref{71} into \eqref{72}, we get
	\begin{eqnarray}\label{73}
	\d^{ac}A&\ge& -2KA+2\left(\frac{1}{\b}-1+\frac{u}{u+\e}\right)\left\langle\nabla A,\nabla \ln w\right\rangle-2\left(\frac{1}{\b}-1+\frac{u}{u+\e}\right)F\left\langle\nabla \Phi,\nabla \ln w\right\rangle\nonumber\\
	&&+C(u+\e)^{\b}\Phi F^2+\Delta^{ac} \Phi\cdot F+2\left\langle\frac{\nabla \Phi}{\Phi} ,\nabla A \right\rangle-2A\cdot\frac{|\nabla \Phi|^2}{\Phi^2}-\frac{f(L\e)}{\e}A.
	\end{eqnarray}	
	Cauchy-Schwarz inequality and basic inequality again give	
	\begin{eqnarray}\label{7cs}
-2\left(\frac{1}{\b}-1+\frac{u}{u+\e}\right)F\left\langle\nabla \Phi,\nabla \ln w\right\rangle	&\ge&-\frac{C}{2}\Phi F\frac{|\nabla w|^2}{w^2}-\frac{2}{\b^2 C}F\frac{|\nabla \Phi|^2}{\Phi}\nonumber\\
	&\ge&-\frac{C}{2}(u+\e)^{\b}\Phi F^2-\frac{2}{\b^2 C}A\frac{|\nabla \Phi|^2}{\Phi^2},
	\end{eqnarray}
	and for any $\delta>0$,
	\begin{eqnarray}\label{7cs2}
	2\left(\frac{1}{\b}-1+\frac{u}{u+\e}\right)\left\langle\nabla A,\nabla \ln w\right\rangle&\ge&	-\frac{1}{\b^2\delta}(u+\e)^{\b}|\nabla A|^2-\delta(u+\e)^{-\b}\frac{|\nabla w|^2}{w^2}\nonumber\\
	&\ge&-P(u,\b,\e,\delta)|\nabla A|^2-\delta F,
	\end{eqnarray}
	where $$P(u,\b,\e,\delta)=\frac{1}{\b^2\delta}\left(\sup\limits_{B(x_0,2R)}u+\e\right)^{\b}.$$ Combining \eqref{73}, \eqref{7cs} and \eqref{7cs2}, we get
	\begin{eqnarray}
	\d^{ac}A+\left\langle\nabla A,\nabla B\right\rangle
	&\ge& \frac{C}{2}(u+\e)^{\b}\Phi F^2-\left(2K+\frac{f(L\e)}{\e}\right)A\nonumber\\
	&&+(\Delta^{ac} \Phi-\delta) F-2\left(1+\frac{1}{\b^2 C}\right)A\frac{|\nabla \Phi|^2}{\Phi^2},
	\end{eqnarray}
	where $B=-2\ln \Phi+P(u,\b,\e,\delta)A$.

Notice that $A$ achieves its strict maximum in $B(x_0,\frac{5}{4}R)$ in the sense of Lemma \ref{MMP}. Therefore, by Lemma \ref{MMP}, we have a sequence $\{x_j\}_{j=j_0}^{\infty}\subset B(x_0,\frac{3}{2}R)$ such that
\begin{equation}
A_j=A(x_j)\ge \sup\limits_{B(x_0,\frac{3}{2}R)}A-\frac{1}{j}>0
\end{equation}
and 
\begin{eqnarray}\label{733}
\frac{1}{j}\ge\frac{C}{2}\e^{\b}\Phi F^2_j+(\Delta^{ac} \Phi-\delta) F_j-\left(2K+\frac{f(L\e)}{\e}\right) A_j-2\left(1+\frac{1}{\b^2 C}\right)\frac{|\nabla \Phi|^2}{\Phi^2}A_j,
\end{eqnarray}
where $F_j:=F(x_j)$. Multiplying $\Phi(x_j)$ on both sides of \eqref{733} gives
\begin{eqnarray}\label{734}
\frac{1}{j}\ge\frac{C}{2}\e^{\b} A^2_j+ (\Delta^{ac} \Phi-\delta) A_j-\left(2K+\frac{f(L\e)}{\e}\right)\Phi A_j-2\left(1+\frac{1}{\b^2 C}\right)\frac{|\nabla \Phi|^2}{\Phi}A_j.
\end{eqnarray}	
  Using property of $\Phi$ and letting $j\to \infty$ first and then letting $\delta\to 0^{+}$, we see	
\begin{equation}
\sup\limits_{B(x_0,\frac{3}{2}R)}A\le C(N,\Lambda)\e^{-\b}\left(K+\frac{1}{R^2}+\frac{f(L\e)}{\e}\right),
\end{equation}
which yields
\begin{equation}\label{736}
\sup\limits_{B(x_0,R)}(u+\e)^{-\b}\left(\frac{|\nabla u|^2}{u^2}+\frac{f(u)}{u}\right)\le C(N,\Lambda)\e^{-\b}\left(K+\frac{1}{R^2}+\frac{f(L\e)}{\e}\right).
\end{equation}
Then we completed the proof.
\end{proof}

Via Theorem \ref{main2} ($N\ge4$ case) and Definition \ref{d0}, we immediately get a proof of Theorem \ref{S} in the case $N\ge4$.
\begin{proof}[Proof of Theorem \ref{S} in the case $N\ge 4$]
By Remark \ref{7r1}, we can choose 
\begin{equation}
	\b_0=\frac{1}{2}\left(\rho(N,\Lambda)+\min\left(\b,\frac{2}{N-2}\right)\right),
\end{equation}
and inequality \eqref{736} holds for such $\b_0$, where $\rho(N,\Lambda)$ is defined as in Theorem \ref{S}. Hence Theorem \ref{main2} ($N\ge 4$) can be read as
\begin{equation}\label{7360}
	\sup\limits_{B(x_0,R)}(u+\e)^{-\b_0}\left(\frac{|\nabla u|^2}{u^2}+\frac{f(u)}{u}\right)\le C(N,\Lambda)\e^{-\b_0}\left(K+\frac{1}{R^2}+\frac{f(L\e)}{\e}\right).
\end{equation}

If $\lim\limits_{t\to 0^+}t^{-1}f(t)=0$ and $t^{-1-\b_0}f(t)$ is non-decreasing on $(0,\infty)$ (with $\b_0\in(\rho(N,\Lambda),\frac{2}{N-2})$)\footnote{Notice that $t^{-1-\b}f(t)$ is $h$-inverse bounded implies $t^{-1-\b_0}f(t)$ is $h$-inverse bounded (hence is non-decreasing) because $f(t)>0$ and $\b\ge\b_0$.}, then we can choose $\e>0$ such that
\begin{equation}\label{ed}
	\frac{f(L\e)}{L\e}=K+\frac{1}{R^2}.
\end{equation}

 For $\mu$-a.e. $y\in B(x_0,R)$, if $u(y)\le\e$, \eqref{7360} gives
\begin{equation}\label{737}
	\left(\frac{|\nabla u|^2}{u^2}+\frac{f(u)}{u}\right)(y)\le C(N,\Lambda)\left(K+\frac{1}{R^2}\right),
\end{equation}
if $u(y)>\e$, then \eqref{7360} gives
\begin{equation}
	\frac{f(u)}{u^{1+\b_0}}(y)\le C(N,\Lambda)\frac{f(L\e)}{(L\e)^{1+\b_0}},
\end{equation}
which yields that $u(y)\le C(N,\Lambda)\e$ because $t^{-1-\b_0}f(t)$ is $h$-inverse bounded. By \eqref{7360} again, we also get 
\begin{equation}\label{738}
\left(\frac{|\nabla u|^2}{u^2}+\frac{f(u)}{u}\right)(y)\le C(N,\Lambda,\b)\left(K+\frac{1}{R^2}\right),
\end{equation}
Combining \eqref{737} and \eqref{738}, we completed the proof.
	
\end{proof}

In order to ensure that the coefficients on the right-hand side of \eqref{k22} are positive, we use the second kind auxiliary function $G$. The advantage of this approach is that the coefficient of $w^{-4}|\nabla w|^4$ does not depend on $\epsilon$. However, a clear disadvantage is that we require $\lambda>2$.
\begin{lem}\label{ac}
	Let $N\in[1,4)$, $\e>0$, $\g=1$, and $X$, $Y$, $Z$ be defined as in Lemma \ref{k2}. We also assume $\lambda>2$, then there exist constants $\b=\b(N,\Lambda)\in(0,\infty)$ if $N\in[1,2]$ or  $\b=\b(N,\Lambda)\in(0,\frac{2}{N-2})$ if $N\in(2,4)$, $d=d(N,\Lambda)>0$ and $L=L(N,\l,\Lambda)>0$ such that for any $\e>0$ 
	\begin{equation}
	X\ge X_0>0,\nonumber
	\end{equation}
	\begin{equation}
	Y\ge Y_0-L\chi_{\{u< L\epsilon\}},\nonumber
	\end{equation}
	\begin{equation}
	Z\ge Z_0-L\chi_{\{u< L\epsilon\}},\nonumber
	\end{equation}
	where $X_0,Y_0,Z_0$ are positive constants which only depend on $N, \Lambda$.	
	
\end{lem}
\begin{proof}
	As before, if $\g=1$, we have
	\begin{equation}
	X_0:=X=\left(\frac{2}{N}\left(1+\frac{1}{\b}\right)-1\right)\left(1+\frac{1}{\b}\right)>0,\nonumber
	\end{equation}
	if $N\le 2$ and $\b>0$ or $N\in(2,4)$ and $\b\in(0,\frac{2}{N-2})$.
	By \eqref{7k}, we can see 
	\begin{eqnarray}
	Y&\ge&\left(\frac{4}{N}(1+\b)+2+\b\right)\frac{u}{u+\e}-2\frac{uf'}{f}+\frac{2d}{\b}\left(\frac{1}{\b}\frac{u+\e}{u}\left(\frac{uf'}{f}-1\right)\right)\nonumber\\
	&&+d\left(\frac{1}{\b^2}\frac{(u+\e)^2}{u^2}\left(\frac{uf'}{f}-1\right)\left(\frac{uf'}{f}-2\right)+1+\frac{1}{\b}-\frac{2}{\b}\frac{u+\e}{u}\left(\frac{uf'}{f}-1\right)\right)\nonumber\\
	&=&Y(f)-\left(\b+2+\frac{4}{N}\left(1+\b\right)\right)\frac{\e}{u+\e}\nonumber\\
	&&+d\left(\frac{1}{\b^2}\frac{2u\e+\e^2}{u^2}\left(\frac{uf'}{f}-1\right)\left(\frac{uf'}{f}-2\right)+\frac{2}{\b}\left(\frac{1}{\b}-1\right)\left(\frac{uf'}{f}-1\right)\frac{\e}{u}\right)\nonumber\\
	&\ge&Y(f)-\left(\b+2+\frac{4}{N}\left(1+\b\right)\right)\frac{\e}{u+\e}+d\left(\frac{1}{\b^2}\frac{2u\e+\e^2}{u^2}\left(\lambda-1\right)\left(\lambda-2\right)-\frac{2\Lambda}{\b}\frac{\e}{u}\right),\nonumber\\
	Z&=&Z(f)-\frac{d\b\e}{u+\e}-\frac{2\b^2}{N}\frac{\e^2+2u\e}{(u+\e)^2},\nonumber
	\end{eqnarray}
	where
	\begin{eqnarray}
	Y(f)&=&\frac{4}{N}(1+\b)+2\left(1-\frac{uf'}{f}\right)+\b+d\left(1+\frac{1}{\b}\left(1-\frac{uf'}{f}\right)\right)\left(1-\frac{1}{\b}\frac{uf'}{f}\right),\nonumber\\
	Z(f)&=&\frac{2}{N}\b^2+d\left(\b+1-\frac{uf'}{f}\right).\nonumber
	\end{eqnarray}
	Now, we divide the following argument into three cases.
	
	Case 1: $N\in[1,2]$. We choose $\b\in\left(\left(\frac{1}{2}+\frac{2}{N}\right)^{-1}\left(\Lambda-1-\frac{2}{N}\right),\infty\right)$ and  $d=d(N,\Lambda)>0$ such that $Y(f)\ge Y_0>0$ and $Z(f)\ge Z_0>0$, where $Y_0, Z_0$ only depend on $N,\Lambda$. Then  there exists $L=L(N,\lambda, \Lambda)>0$ such that
	\begin{eqnarray}
	Y&\ge&\frac{Y_0}{2}-L\chi_{\{u< L\epsilon\}},\nonumber\\
	Z&\ge&\frac{Z_0}{2}-L\chi_{\{u< L\epsilon\}}.\nonumber
	\end{eqnarray}
	Here, we use the condition $\lambda>2$ to ensure  the coefficient of $\frac{\e^2}{u^2}$ is positive and so there exists $L=L(N,\lambda, \Lambda)>0$ such  that
	\begin{equation}
	-\left(\b+2+\frac{4}{n}\left(1+\b\right)\right)\frac{\e}{u+\e}+d\left(\frac{1}{\b^2}\frac{2u\e+\e^2}{u^2}\left(\lambda-1\right)\left(\lambda-2\right)-\frac{2\Lambda}{\b}\frac{\e}{u}\right)\ge -L
	\end{equation}
	and 
	\begin{equation}
	-\left(\b+2+\frac{4}{n}\left(1+\b\right)\right)\frac{\e}{u+\e}+d\left(\frac{1}{\b^2}\frac{2u\e+\e^2}{u^2}\left(\lambda-1\right)\left(\lambda-2\right)-\frac{2\Lambda}{\b}\frac{\e}{u}\right)\ge -\frac{Y_0}{2}
	\end{equation}
	on the set $\{x:u(x)\ge L\e\}$.

	Case 2:	$N\in(2,3)$ and $\Lambda<\frac{N+1}{N-2}$. As in Case 1, we choose $\b\in\left(\left(\frac{1}{2}+\frac{2}{N}\right)^{-1}\left(\Lambda-1-\frac{2}{N}\right),\frac{2}{N-2}\right)$ and suitable $d=d(N,\Lambda)>0$ such that $Y(f)\ge Y_0>0$ and $Z(f)\ge Z_0>0$, where $Y_0, Z_0$ only depend on $N,\Lambda$. The remaining part is same as Case 1.

	Case 3: $N\in(2,3)$ and $\Lambda\in[\frac{N+1}{N-2},p_S(N))$ or $N\in[3,4)$ (notice $\Lambda\ge p(N)$). As proof of Lemma \ref{ab}, we can choose $d_0=\frac{2\b^2}{N(\Lambda-1-\b)}$, where $\b\in(0,\frac{2}{N-2})$ is undetermined. Define
	$Q(x,\b,d)$ as proof of Lemma \ref{ab},
	where $d\in(0,d_0]$ and $\b\in(0,\min(\frac{2}{N-2},\Lambda-\frac{N-1}{N-2}))$ is arbitrary, and we know
	\begin{equation}
	\text{the symmetry axis of}\,\,\, Q(\cdot,\b,d)=\frac{1}{2}+\b+\frac{\b^2}{d}\ge \Lambda.
	\end{equation}
	 Hence, $Y(f)\ge Q(\Lambda,\b,d)$. If we choose $d=d_0$, then there exists $\b=\b(N,\Lambda)\in(0,\frac{2}{N-2})$ (any $\b\in\left(2\left(\frac{N-1}{N+2}\Lambda-1\right),\frac{2}{N-2}\right)$ is satisfied) such that $Q(\Lambda,\b,d_0)>0$. For this fixed $\b$, by continuity of $Q$, we can fix a $d=d(N,\Lambda)\in(0,d_0)$ such that $Q(\Lambda,\b,d)>0$. So, for these fixed $\b,d$, $Y(f)\ge Y_0>0$, $Z(f)\ge Z_0>0$,  where $Y_0, Z_0$ only depend on $N,\Lambda$. Then, by same argument as Case 1, we finish the proof.
	
\end{proof}
\begin{rmk}\label{7r2}
\rm	Actually, when $N\in(2,4)$, any $\b\in\left(2\left(\frac{N-1}{N+2}\Lambda-1\right),\frac{2}{N-2}\right)$ is satisfied if $\Lambda\ge\frac{N+1}{N-2}$ or any $\b\in\left(2\left(\frac{N-1}{N+2}\Lambda-1\right),\min(\frac{2}{N-2},\Lambda-\frac{N-1}{N-2})\right)$ (not empty) is satisfied if $N\in[3,4)$.
\end{rmk}

Now, we can prove the Theorem \ref{main2} in  the case $N\in[1,4)$.

\begin{proof}[Proof of Theorem \ref{main2} in the case $N\in[1,4)$]
	
	By Lemma \ref{k2} and Lemma \ref{ac}, for $\g=1$ in the construction of $G$, we can choose positive numbers $\b=\b(N,\Lambda)$, $d=d(N,\Lambda)$, $L=L(N,\l,\Lambda)$ and $C=C(N,\Lambda)$ such that for any $\e>0$,
	\begin{equation}\label{711}
	\d G\ge-2KG+\frac{2}{\b}\left\langle\nabla G,\nabla\ln w\right\rangle+Cw^{-1}G^2-\frac{f(L\e)}{\e}G
	\end{equation}
	on $B(x_0,\frac{3}{2}R)$, because $t^{-1}f(t)$ is non-decreasing on $(0,\infty)$ (by $\lambda>1$).
	
	As before, we can obtain the boundedness of $G$ via \eqref{711} as follows.
	
	Let $A:=\Phi G$ be the auxiliary function, where $\Phi$ is an undetermined cut-off function as Lemma \ref{cut off}
	by choosing $\alpha=\frac{M_1}{2M_2}$, where $M_1=\sup\limits_{B(x_0,R)} A>0$ (or proof is finished) and $M_2=\sup\limits_{B(x_0,\frac{3}{2}R)} A$.   First, as proof of Theorem \ref{main2} in the case $N\ge 4$, we get $\Delta^{s}A\ge 0$
	on $B(x_0,\frac{3R}{2})$.
	By chain rule and $\Phi\ge\alpha$, as before, we  have	
	\begin{eqnarray}\label{721}
	\d^{ac}A=\Delta^{ac} \Phi\cdot G+\d^{ac}G\cdot \Phi+2\left\langle\frac{\nabla \Phi}{\Phi} ,\nabla A \right\rangle-2A\cdot\frac{|\nabla \Phi|^2}{\Phi^2}.
	\end{eqnarray}
	Substituting \eqref{711} into \eqref{721}, we get
	\begin{eqnarray}\label{731}
	\d^{ac}A&\ge& -2KA+\frac{2}{\b}\left\langle\nabla A,\nabla \ln w\right\rangle-\frac{2}{\b}G\left\langle\nabla \Phi,\nabla \ln w\right\rangle\nonumber\\
	&&+Cw^{-1}\Phi G^2+\Delta^{ac} \Phi\cdot G+2\left\langle\frac{\nabla \Phi}{\Phi} ,\nabla A \right\rangle-2A\cdot\frac{|\nabla \Phi|^2}{\Phi^2}-\frac{f(L\e)}{\e}A.
	\end{eqnarray}	
	Cauchy-Schwarz inequality and basic inequality again give	
	\begin{eqnarray}\label{7cs1}
	-\frac{2}{\b}G\left\langle\nabla \Phi,\nabla \ln w\right\rangle	&\ge&-\frac{C}{2}\Phi G\frac{|\nabla w|^2}{w^2}-\frac{2}{\b^2 C}G\frac{|\nabla \Phi|^2}{\Phi}\nonumber\\
	&\ge&-\frac{C}{2}w^{-1}\Phi G^2-\frac{2}{\b^2 C}A\frac{|\nabla \Phi|^2}{\Phi^2}.
	\end{eqnarray}
	Combining \eqref{731} and \eqref{7cs1}, we get
	\begin{eqnarray}
		\d^{ac}A+\left\langle\nabla A,\nabla B\right\rangle\ge \frac{C}{2}w^{-1}\Phi G^2-\left(2K+\frac{f(L\e)}{\e}\right)A+\Delta^{ac} \Phi\cdot G-2\left(1+\frac{1}{\b^2 C}\right)A\frac{|\nabla \Phi|^2}{\Phi^2},\nonumber
	\end{eqnarray}
	where $B=-2\ln \Phi-\frac{2}{\b}\ln w$.
	
	Notice that $A(x)$ achieves its strict maximum in $B(x_0,\frac{5}{4}R)$ in the sense of Lemma \ref{MMP}. Therefore, by Lemma \ref{MMP}, we have a sequence $\{x_j\}_{j=j_0}^{\infty}\subset B(x_0,\frac{3}{2}R)$ such that
	\begin{equation}
	A_j=A(x_j)\ge \sup\limits_{B(x_0,\frac{3}{2}R)}A-\frac{1}{j}>0
	\end{equation}
	and 
	\begin{eqnarray}\label{7331}
	\frac{1}{j}\ge\frac{C}{2}\e^{\b}\Phi G^2(x_j)+\Delta^{ac} \Phi\cdot G(x_j)-\left(2K+\frac{f(L\e)}{\e}\right) A_j-2\left(1+\frac{1}{\b^2 C}\right)\frac{|\nabla \Phi|^2}{\Phi^2}A_j.
	\end{eqnarray}
	Multiplying $\Phi(x_j)$ on both sides of \eqref{7331},
	 letting $j\to \infty$ and using property of $\Phi$, we obtain	
	\begin{equation}
	\sup\limits_{B(x_0,\frac{3}{2}R)}A\le C(N,\Lambda)\e^{-\b}\left(K+\frac{1}{R^2}+\frac{f(L\e)}{\e}\right),
	\end{equation}
	which yields
	\begin{equation}\label{7361}
	\sup\limits_{B(x_0,R)}(u+\e)^{-\b}\left(\frac{|\nabla u|^2}{(u+\e)^2}+\frac{f(u)}{u}\right)\le C(N,\Lambda)\e^{-\b}\left(K+\frac{1}{R^2}+\frac{f(L\e)}{\e}\right).
	\end{equation}
We completed the proof.
\end{proof}

\begin{proof}[Proof of Theorem \ref{S} in the case $N\in[1,4)$]
	By Remark \ref{7r2}, we choose 
	\begin{equation}
		\b_0=\frac{1}{2}\left(\rho(N,\Lambda)+\min\left(\b,\frac{2}{N-2},\Lambda-\frac{N-1}{N-2}\right)\right),
	\end{equation}
	and we know that Theorem \ref{main2} ($N<4$) can be read as
	\begin{equation}\label{73600}
	\sup\limits_{B(x_0,R)}(u+\e)^{-\b_0}\left(\frac{|\nabla u|^2}{(u+\e)^2}+\frac{f(u)}{u}\right)\le C(N,\Lambda)\e^{-\b_0}\left(K+\frac{1}{R^2}+\frac{f(L\e)}{\e}\right).
\end{equation}
By same reason as case $N\ge 4$, we can choose $\e>0$ such that $\frac{f(L\e)}{L\e}=K+R^{-2}$. For $\mu$-a.e. $y\in B(x_0,R)$, if $u(y)\le\e$, \eqref{73600} gives
	\begin{equation}\label{7371}
	\frac{f(u)}{u}(y)\le C(N,\l,\Lambda)\left(K+\frac{1}{R^2}\right),
	\end{equation}
	if $u(y)>\e$, then
	\begin{equation}
	\frac{f(u)}{u^{1+\b_0}}(y)\le C(N,\l,\Lambda)\frac{f(L\e)}{(L\e)^{1+\b_0}},
	\end{equation}
	which yields that $u(y)\le C(N,
	\l,\Lambda,\b)\e$ because $t^{-1-\b_0}f(t)$ is $\max(1,h)$-inverse bounded. By \eqref{73600} again, we also get \eqref{7371}. Hence $$\sup\limits_{B(x_0,R)}\frac{f(u)}{u}\le C(N,\l,\Lambda)\left(K+\frac{1}{R^2}\right).$$ By Lemma \ref{l31}, we know that universal boundeness estimate $\kappa$-implies logarithmic gradient estimate for $\kappa>1$\footnote{One can get universal boundedness estimate on $B(x_0,\frac{3}{2}R)$ by modifying the property of cut-off function and choose $\kappa=\frac{3}{2}$.}, and we  completed the proof.
	
\end{proof}

\section{Further results for Lichnerowicz type equation}

In this section, we focus on the Lichnerowicz type equation on \RCDn\,spaces with $K\ge 0$ and $N\ge1$ and aim to strengthen the result stated in Theorem \ref{main1}. Specifically, we will prove the logarithmic gradient estimate of Lichnerowicz-type equation and consequently obtain a Liouville theorem for it even when $K>0$. We write the equation as follows.
\begin{equation}\label{leq}
	\d u+au-bu^{\sigma}+cu^{\tau}=0,
\end{equation}
where $\sigma>1$, $\tau<1$ and $a\ge0,b\ge0,c\ge0$.

As a special case of Lemma \ref{k4}, we immediately get the following lemma.
\begin{lem}\label{k8}
	Let $u$ be a positive solution of \eqref{leq} on $B(x_0,2R)$. The function $w$ is defined by equation \eqref{tran} and $F$ is its first kind auxiliary function as \eqref{F} with $\e=\g=d=0$. Then 
	\begin{eqnarray}\label{81}
	\Delta F&\ge&2\left(\frac{1}{\b}-1\right)\left\langle\nabla F,\nabla \ln w\right\rangle+\left(\frac{2}{N}\left(1+\frac{1}{\b}\right)^2-2\right)F^2+\frac{2\b^2}{N}\left(a-bu^{\sigma-1}+cu^{\tau-1}\right)^2\nonumber\\
	&&+\left(\left(\frac{4}{N}(1+\b)+2\right)\left(a-bu^{\sigma-1}+cu^{\tau-1}\right)-2\left(a-b\sigma u^{\sigma-1}+c\tau u^{\tau-1}\right)-2K\right)F\nonumber
	\end{eqnarray}
	on $B(x_0,\frac{3}{2}R)$.	
\end{lem}

By Lemma \ref{k8}, we readily get the following elliptic inequality, which is crucial for deriving the Liouville theorem of equation \eqref{leq} on negatively curved spaces.

\begin{lem}\label{k9}
	Let $u$ be a positive solution of \eqref{leq} on $B(x_0,2R)$, $w$ be defined by \eqref{tran} and $F$ be its first kind auxiliary function as \eqref{F} with $\e=\g=d=0$. Then for any $\delta\in(0,\frac{4}{\sqrt{N}(\sqrt{N}-1)})$\footnote{Here, we set $\frac{4}{\sqrt{N}(\sqrt{N}-1)}=\infty$ when $N=1$.}, there exists $\b=\b(N,\sigma,\delta)>0$ such that 
	\begin{eqnarray}\label{91}
	\Delta F\ge2\left(\frac{1}{\b}-1\right)\left\langle\nabla F,\nabla \ln w\right\rangle+MF^2+\left(L_{abc}(N,\sigma,\tau,\delta)-2K\right)F\nonumber
	\end{eqnarray}
	on $B(x_0,\frac{3}{2}R)$, where $M=M(N,\sigma)>0$ and
	\begin{eqnarray}
		L_{abc}(N,\sigma,\tau,\delta)=\begin{cases}
		2(\sigma-1)a\qquad\qquad\qquad\,\,\text{if} \qquad \sigma\in(1,1+\frac{2}{\sqrt{N}(\sqrt{N}-1)})\\
		\delta a \qquad\qquad\qquad\qquad\quad\,\,\,\, \text{if}\qquad\sigma\in[1+\frac{2}{\sqrt{N}(\sqrt{N}-1)},\infty).
		\end{cases}
	\end{eqnarray}		
\end{lem}
\begin{rmk}
\rm From the following proof, we can see that if $\sigma\in(1,1+\frac{2}{\sqrt{N}(\sqrt{N}-1)})$, then $\b$ does not depend on $\delta$; if $\sigma\in[1+\frac{2}{\sqrt{N}(\sqrt{N}-1)},\infty)$ and $\delta\to \frac{4}{\sqrt{N}(\sqrt{N}-1)}$, then $\b(N,\sigma,\delta)\to \frac{1}{\sqrt{N}-1}$.
\end{rmk}
\begin{proof}
	When $\sigma\in(1,1+\frac{2}{N}]$,  then combining Lemma \ref{k8} and basic inequality give
	\begin{eqnarray}\label{9111}
	\Delta F&\ge&2\left(\frac{1}{\b}-1\right)\left\langle\nabla F,\nabla \ln w\right\rangle+2F^2+\frac{4}{N}\sqrt{(1+\b)^2-2N\b^2}|a-bu^{\sigma-1}+cu^{\tau-1}|\nonumber\\
	&&+\left(\left(\frac{4}{N}(1+\b)+2\right)\left(a-bu^{\sigma-1}+cu^{\tau-1}\right)-2\left(a-b\sigma u^{\sigma-1}+c\tau u^{\tau-1}\right)-2K\right)F\nonumber\\
	&&
	\end{eqnarray}
	on $B(x_0,\frac{3}{2}R)$.	We fix $\b\in(0,\frac{1}{\sqrt{2N}-1})$ such that 
	\begin{equation}\label{9112}
	\frac{4}{N}(1+\b)-\frac{4}{N}\sqrt{(1+\b)^2-2N\b^2}=2(\sigma-1).
	\end{equation}
	Combining \eqref{9111} and \eqref{9112}, we obtain the desired result because $\tau<1$.
	
When $\sigma\in(1+\frac{2}{N},1+\frac{2}{\sqrt{N}(\sqrt{N}-1)})$, we can choose $\b=\frac{N}{2}(\sigma-1)-1$, then Lemma \ref{k8} yields
	\begin{eqnarray}\label{92}
	\Delta F\ge2\left(\frac{1}{\b}-1\right)\left\langle\nabla F,\nabla \ln w\right\rangle+\left(\frac{2}{N}\left(1+\frac{1}{\b}\right)^2-2\right)F^2+\left(2(\sigma-1)a-2K\right)F\nonumber
	\end{eqnarray}
	on $B(x_0,\frac{3}{2}R)$.
		
When $\sigma\in[1+\frac{2}{\sqrt{N}(\sqrt{N}-1)},\infty)$, for any $\delta\in(0,\frac{4}{\sqrt{N}(\sqrt{N}-1)})$, we can choose $\b\in(\frac{1}{N},\frac{1}{\sqrt{N}-1})$ such that $\frac{4}{N}(1+\b)>\delta$.	Then using Lemma \ref{k8} also derives the desired result.

\end{proof}

At the aid of Lemma \ref{k9}, with same argument as before, we have
\begin{thm}\label{t8}
Let $u\in\W(B(x_0,2R))\cap L^{\infty}(B(x_0,2R))$ be a positive solution of \eqref{leq} on $B(x_0,2R)$, then for any $\delta\in(0,\frac{4}{\sqrt{N}(\sqrt{N}-1)})$, we have
\begin{equation}
	\sup\limits_{B(x_0,R)} \frac{|\nabla u|^2}{u^2}\le C\left(\frac{1}{R^2}+\frac{\sqrt{K}}{R}+\left(2K-L_{abc}(N,\sigma,\tau,\delta)\right)^{+}\right),
\end{equation}
 where $C=C(N,\sigma)>0$	and $L_{abc}(N,\sigma,\tau,\delta)$ is defined as in Lemma \ref{k9}.
\end{thm}
\begin{proof}

We choose fixed $\b=\b(N,\Lambda,\delta)>0$ and $M=M(N,\sigma)>0$ such that Lemma \ref{k9} holds.  
Let us define $A:=\Phi F$ as the auxiliary function, where $F$ is  same as in Lemma \ref{k8}, and $\Phi$ is an undetermined cut-off function as defined in Lemma \ref{cut off} with $\alpha=\frac{M_1}{2M_2}$, where $M_1=\sup\limits_{B(x_0,R)} A>0$ (or the proof is finished) and $M_2=\sup\limits_{B(x_0,\frac{3}{2}R)} A$. By Radon--Nikodym decomposition, we see
$\Delta^{s}A\ge 0$
on $B(x_0,\frac{3R}{2})$, because the singular part of $\d \Phi$ and $\d F$ are non-negative by Lemma \ref{cut off} and Lemma \ref{k9}.
By $\Phi\ge\alpha$ and chain rule, we have
\begin{eqnarray}\label{9aeq}
\d^{ac}A=\Delta^{ac} \Phi\cdot F+\d^{ac}F\cdot \Phi+2\left\langle\frac{\nabla \Phi}{\Phi} ,\nabla A \right\rangle-2A\cdot\frac{|\nabla \Phi|^2}{\Phi^2}.
\end{eqnarray}
By Lemma \ref{k9} and chain rule  again, we see
\begin{eqnarray}\label{912}
\d^{ac}A&\ge& 2\left(\frac{1}{\b}-1\right)\left\langle\nabla A,\nabla \ln w\right\rangle+M\Phi F^2-\left(2K-L_{abc}(N,\sigma,\tau,\delta)\right)^{+}F\nonumber\\
&&-2\left(\frac{1}{\b}-1\right)F\left\langle\nabla \Phi,\nabla \ln w\right\rangle+\Delta^{ac} \Phi\cdot F+2\left\langle\frac{\nabla \Phi}{\Phi} ,\nabla A \right\rangle-2A\frac{|\nabla \Phi|^2}{\Phi^2}.
\end{eqnarray}
Cauchy-Schwarz inequality and basic equality provide
\begin{eqnarray}\label{9bb}
-2\left(\frac{1}{\b}-1\right)F\left\langle\nabla \Phi,\nabla \ln w\right\rangle\ge-\left(\frac{1}{\b}-1\right)^2\frac{2}{M}\frac{|\nabla\Phi|^2}{\Phi}F-\frac{M}{2}\Phi F^2.
\end{eqnarray}
Then substituting \eqref{9bb} into \eqref{912} yields
\begin{eqnarray}
&&\d^{ac}A+\left\langle\nabla A,\nabla B\right\rangle\nonumber\\
&&\ge \Delta^{ac} \Phi\cdot F+\frac{M}{2}\Phi F^2-2\left(\frac{1}{M}\left(\frac{1}{\b}-1\right)^2+1\right)\frac{|\nabla \Phi|^2}{\Phi^2}A-\left(2K-L_{abc}(N,\sigma,\tau,\delta)\right)^{+}F,\nonumber
\end{eqnarray}
where $B=2\left(1-\frac{1}{\b}\right)\ln w-2\ln \Phi$. Notice that $A$ achieves its strict maximum in $B(x_0,\frac{5}{4}R)$ in the sense of Lemma \ref{MMP}. Therefore, by Lemma \ref{MMP}, we have a sequence $\{x_j\}_{j=j_0}^{\infty}\subset B(x_0,\frac{3}{2}R)$ such that 
\begin{equation}
A_j=A(x_j)\ge \sup\limits_{B(x_0,\frac{3}{2}R)}A-\frac{1}{j}>0
\end{equation}
and 
\begin{equation}\label{943}
\frac{1}{j}\ge \Delta^{ac} \Phi\cdot F_j+\frac{M}{2}\Phi F^2_j-2\left(\frac{1}{M}\left(\frac{1}{\b}-1\right)^2+1\right)\frac{|\nabla \Phi|^2}{\Phi^2}A_j-\left(2K-L_{abc}(N,\sigma,\tau,\delta)\right)^{+}F_j,
\end{equation}
where $F_j=F(x_j)$. Multiplying $\Phi(x_j)$ on both sides of \eqref{943}, using the property of $\Phi$ and letting $j\to\infty$ in \eqref{943}, we obtain
\begin{equation}
\sup\limits_{B(x_0,\frac{3}{2}R)}A\le C(N,\sigma)\left(\left(\frac{1}{M}\left(\frac{1}{\b}-1\right)^2+1\right)\left(\frac{\sqrt{K}}{R}+\frac{1}{R^2}\right)+\left(2K-L_{abc}(N,\sigma,\tau,\delta)\right)^{+}\right),\nonumber
\end{equation}
which implies
\begin{equation}\label{919}
\sup\limits_{B(x_0,R)}\frac{|\nabla u|^2}{u^2}\le \frac{C(N,\sigma)}{\b^2}\left(\left(\frac{1}{\b^2}+1\right)\left(\frac{\sqrt{K}}{R}+\frac{1}{R^2}\right)+\left(2K-L_{abc}(N,\sigma,\tau,\delta)\right)^{+}\right).
\end{equation}
Notice that $\b$ only depends on $N,\sigma$ if $\sigma<1+\frac{2}{\sqrt{N}(\sqrt{N}-1)}$ and $\b>\frac{1}{N}$ if $\sigma\ge1+\frac{2}{\sqrt{N}(\sqrt{N}-1)}$ (see proof of Lemma \ref{k9}). So we finish the proof.

\end{proof}
For any $a\ge0$, $n\ge 1$ and $\sigma>1$, we define
\begin{eqnarray}\label{con}
	L(n,a,\sigma)=\begin{cases}
		0\qquad\qquad\qquad\quad\text{if}\quad a=0,\\
		(\sigma-1)a\qquad\quad\,\,\,\,\text{if}\quad a>0,\sigma\in(1,1+\frac{2}{\sqrt{n}(\sqrt{n}-1)}),\\
		\frac{2a}{\sqrt{n}(\sqrt{n}-1)}\qquad\quad\,\,\,\,\text{if}\quad a>0,\sigma\in[1+\frac{2}{\sqrt{n}(\sqrt{n}-1)},\infty).
	\end{cases}
\end{eqnarray}
The following Liouville theorem is direct corollary of Theorem \ref{t8}.

\begin{thm}\label{t81}
	Let $u$ be a positive solution of \eqref{leq} on \RCDn\,metric measure space $X$ and $L(N,a,\sigma)$ is defined by \eqref{con}, if $K\le L(N,a,\sigma)$, then $u$ is constant.
\end{thm}

\begin{rmk}
	\rm From Theorem \ref{t81}, we immediately see that the global classical positive solution of Allen-Cahn equation ($a=b=1,c=0,\sigma=3$) or static Fisher-KPP equation ($a=b=1,c=0,\sigma=2$) on hyperplane $\mathbb{H}^2$ must be constant $1$. We  suspect that the constant $L(N,a,\sigma)$ may not be optimal in the above Liouville theorem. If one improves this bound, then one obtains the Liouville theorem for \eqref{leq} on more negatively curved spaces.
\end{rmk}

\section{Appendix}\label{S9}
We provide an illustrative example to demonstrate the optimality of Theorem \ref{sthm} in certain cases. For \RCDO\,spaces, Theorem \ref{sthm} is confirmed as optimal through the classical singular solutions of Lane-Emden equation on the Euclidean spaces (see \cite{GS,PQS,SZ}). Therefore, we examine the solutions of \eqref{LE} on a \RCDn\,space with $K>0$. For any fixed real numbers $N>3$ and $K>0$, we will construct a \RCDn\,space and a global positive solution of \eqref{LE} on the space. We then verify that this solution satisfies the equality case in Theorem \ref{sthm} if we ignore the sharpness of the constant $C(N,\alpha)$.

Define 
\begin{equation}
	n=\sup\{k\in\mathbb{N}:k<N\},\nonumber
\end{equation}
\begin{equation}
	\Gamma(n,\a)=\frac{1}{2}\left(n-2-\frac{4}{\a-1}\right),
\end{equation}
and for an undetermined $\mu>0$, we define
\begin{equation}
	f(x)=\Gamma(n,\a)\ln(\mu^2+|x|^2).
\end{equation}

The metric measure space is defined as $(\mathbb{R}^n,d,e^{-f}dx)$, where $d$ is the standard metric in $\mathbb{R}^n$ and $dx$ is Lebesgue measure. In this smooth metric measure space, we know that 
\begin{equation}
	\d=\d_{\mathbb{R}^n}-\left\langle\nabla f,\nabla\cdot\right\rangle.
\end{equation}

For any $\a\in(1,\frac{N+2}{N-2})$ (hence $\a\in(1,\frac{n+2}{n-2})$), one can verify that
\begin{equation}
	u(x)=\left(\frac{\mu}{\mu^2+|x|^2}\sqrt{\frac{4n}{\a-1}}\right)^{\frac{2}{\a-1}}\nonumber
\end{equation}
satisfies the equation
\begin{equation}
	\d u+u^{\alpha}=0 \qquad\text{in $\mathbb{R}^n$}.
\end{equation}

Then, we calculate the Bakry-\'{E}mery Ricci curvature of the aforementioned metric measure space:

\begin{eqnarray}
	Ric^{N}_{f}:&=&Ric+\nabla^2f-\frac{1}{N-n}df\otimes df\nonumber\\
	&=&\frac{\Gamma(n,\a)}{(\mu^2+|x|^2)^{2}}\left[2\delta_{ij}(\mu^2+|x|^2)-4\left(1+\frac{\Gamma(n,\a)}{N-n}\right)x_i x_j\right]dx^{i}dx^{j}.
\end{eqnarray}
A basic computation shows that the maximal lower bound of $Ric^{N}_{f}$ is given by the following formula (notice that $\Gamma(n,\a)<0$ at present case).
\begin{eqnarray}
	\min_{x\in\mathbb{R}^n}mRic^N_f(x)=
	\begin{cases}
		\frac{2\Gamma(n,\a)}{\mu^2}\qquad\qquad\qquad\qquad\quad\text{if $N-n\ge-\frac{2}{3}\Gamma(n,\a)$}\nonumber\\
		-\frac{\Gamma(n,\a)(N-n+2\Gamma(n,\a))^2}{4(N-n)(N-n+\Gamma(n,\a))\mu^2}\qquad\,\,\text{if $N-n\le-\frac{2}{3}\Gamma(n,\a)$},
	\end{cases}
\end{eqnarray}
where $mRic^N_f(x)$ represents the minimal eigenvalue of $Ric^N_f(x)$. 
Now, we can choose $\mu>0$ such that 
\begin{equation}
	-K=	\min_{x\in\mathbb{R}^n}mRic^N_f.
\end{equation}
 According to \cite[Proposition 4.21]{EKS}, we can conclude that this metric measure space is an \RCDn\,space but not a ${\rm RCD^{*}}(-(K-\e),N)$ space for any $\e>0$. And for the solution $u$, we clearly see 
\begin{equation}\label{99}
	\sup\limits_{\mathbb{R}^n}\left(\frac{|\nabla u|^2}{u^2}+u^{\a-1}\right)=C(N,\a)K,
\end{equation}
where $C(N,\alpha)$ is a positive numbere that depends on $N,\a$. Therefore, \eqref{99} illustates that Theorem \ref{sthm} is optimal.

\begin{rmk}
	The above example does not illustrate that Theorem \ref{sthm} is optimal for the local solution of the Lane-Emden equation, which is valid for same equation on Euclidean spaces. We also suspect that Theorem \ref{sthm} is  optimal at $N\in[1,3]$ case.
\end{rmk}

\section*{Acknowledgements}We are  grateful to the referees for the careful reading of our manuscript and for the insightful comments and valuable suggestions, which have greatly helped us to improve the presentation and clarity of our results.  The author is
supported by Scientific Research Startup Project of Jiangsu Normal University (Project
No: 24XFRS051), the General Program of Basic Scientific Research in Institutions of
Higher Education of Jiangsu Province (Grant No: 25KJB110002) and National Natural
Science Foundation of China (Grant No: 12526552).

\section*{Statement}
The corresponding author declares that there is no conflict of interest.
Data sharing is not applicable to this article, as no datasets were generated or analyzed during the current study.

\end{document}